\documentclass[journal]{IEEEtran}
\IEEEoverridecommandlockouts                              % This command is only needed if 
\bstctlcite{IEEEexample:BSTcontrol} 

\usepackage{enumitem,kantlipsum}
\usepackage{lipsum}
\usepackage{mathtools}
\usepackage{cuted}
\usepackage{romannum}
\usepackage[usenames,dvipsnames]{xcolor}
\usepackage{tabularx,booktabs}
\usepackage{setspace}
\usepackage{setspace}
\usepackage{calrsfs}
\DeclareMathAlphabet{\pazocal}{OMS}{zplm}{m}{n}
\usepackage{graphicx,dblfloatfix}
\usepackage{graphics} 
\usepackage{epsfig} 
\usepackage{mathptmx}
\usepackage{caption}
\usepackage{subcaption}
\usepackage{times} 
\usepackage{tikz}
\usetikzlibrary{positioning}
\usepackage{amsmath, amssymb}

\usepackage{amsthm}
\usepackage{amsfonts} 
\usepackage{setspace}
\usepackage{scalefnt}
\usepackage{multirow}
\usepackage{textcomp}
\usepackage[english]{babel}
\usepackage{float} 
\usepackage{algorithmicx}
\usepackage[section]{algorithm}
\usepackage{algpascal}
\usepackage{algc}
\usepackage{algcompatible}
\usepackage{algpseudocode}
\let\oldReturn\Return
\renewcommand{\Return}{\State\oldReturn}
\usepackage{mathtools}
\usepackage{bm}
\usepackage{mathptmx}
\usepackage{times} 
\usepackage{mathtools, cuted}
\usepackage{textcomp}
\usepackage[english]{babel}
\usepackage{rotating}
\usepackage{pgfplots}
\pgfplotsset{compat=1.5}
\usepackage{pgfplotstable}
\usepackage{filecontents}
\usepackage{diagbox}
\usepackage{makecell}
\usepackage{multirow}
\usepackage[nocomma]{optidef}

\newtheorem{theorem}{Theorem}[section]
\newtheorem{lemma}[theorem]{Lemma}

\newtheorem{assumption}[theorem]{Assumption}
\newtheorem{defn}[theorem]{Definition}
\newtheorem{cor}[theorem]{Corollary}
\newtheorem{remark}[theorem]{Remark}
\newtheorem{prop}[theorem]{Proposition}
\newtheorem{prob}[theorem]{Problem}

\newcommand{\A}{\pazocal{A}}

\renewcommand{\S}{\pazocal{S}}

\newcommand\norm[1]{\left\lVert#1\right\rVert}

\renewcommand{\P}{\pazocal{P}}
\newcommand{\T}{\pazocal{T}}

\newcommand{\D}{\pazocal{D}}
\newcommand{\En}{\pazocal{E}}

\newcommand{\C}{\pazocal{C}}

\newcommand{\bs}{{\bar{s}}}
\newcommand{\G}{{\pazocal{G}}}
\newcommand{\sG}{{\scriptscriptstyle{\G}}}
\newcommand{\sA}{{\scriptscriptstyle{A}}}
\newcommand{\sD}{{\scriptscriptstyle{D}}}

\newcommand{\V}{V_\sG}
\newcommand{\PA}{\P_{\sA}}
\newcommand{\PD}{\P_{\sD}}
\renewcommand{\AA}{\A_{\sA}}
\newcommand{\AD}{\A_{\sD}}
\newcommand{\ppD}{\mathbf{P}_{\sD}}

\newcommand{\pD}{\pi_{\sD}}
\newcommand{\pA}{\pi_{\sA}}

\newcommand{\rA}{r_{\sA}}
\newcommand{\rD}{r_{\sD}}

\renewcommand{\t}{\tau}
\graphicspath{{figures/}}
\numberwithin{theorem}{section}
\newcommand\sbullet[1][.5]{\mathbin{\vcenter{\hbox{\scalebox{#1}{$\bullet$}}}}}
\newcommand{\remove}[1]{}
\def \*{\star}
\def \10n{\!\!\!\!\!\!\!\!\!\!}

\graphicspath{{figures/}}
\def\BibTeX{{\rm B\kern-.05em{\sc i\kern-.025em b}\kern-.08em
		T\kern-.1667em\lower.7ex\hbox{E}\kern-.125emX}}

\begin{document}
%\title{\LARGE \bf A Reinforcement Learning  Approach for Dynamic Information Flow Tracking - Advanced Persistent Threat Game}
%\title{\LARGE \bf A Reinforcement Learning-based Dynamic Information Flow Tracking Games for Detecting Advanced Persistent Threats}
\title{\LARGE \bf A Reinforcement Learning  Approach for Dynamic Information Flow Tracking Games for Detecting Advanced Persistent Threats}
\author{Dinuka Sahabandu, Shana Moothedath,\IEEEmembership{~Member, IEEE,} Joey Allen, \\ Linda Bushnell,\IEEEmembership{~Fellow, IEEE,} Wenke Lee,\IEEEmembership{~Senior Member, IEEE,} and Radha Poovendran,\IEEEmembership{~Fellow, IEEE}
 \thanks{D. Sahabandu, S. Moothedath, L. Bushnell, and R. Poovendran are with the Department of Electrical and Computer Engineering, University of Washington, Seattle, WA 98195, USA. \texttt{\{sdinuka, sm15, lb2, rp3\}@uw.edu}.}
 \thanks{J. Allen and W. Lee are with the College of Computing, Georgia Institute of Technology, Atlanta, GA 30332 USA.
\texttt{jallen309@gatech.edu, wenke@cc.gatech.edu}.}
 }

\maketitle
\thispagestyle{empty}
\pagestyle{empty}
\begin{abstract}
Advanced Persistent Threats (APTs) are stealthy, sophisticated, and long-term attacks that threaten the security and privacy of sensitive information. Interactions of APTs with victim system introduce information flows that are recorded in the system logs. Dynamic Information Flow Tracking (DIFT) is a promising detection mechanism for detecting APTs. DIFT taints information flows originating at system entities that are susceptible to an attack, tracks the propagation of the tainted flows, and authenticates the tainted flows at certain system components according to a pre-defined security policy. Deployment of DIFT to defend against APTs in cyber systems is limited by the heavy resource and performance overhead associated with DIFT. Effectiveness of detection by DIFT depends on the false-positives and false-negatives generated due to inadequacy of DIFT's pre-defined security policies to detect stealthy behavior of APTs. In this paper, we propose a resource efficient model for DIFT by incorporating the security costs, false-positives, and false-negatives associated with DIFT. Specifically, we develop a game-theoretic framework and provide an analytical model of DIFT that enables the study of trade-off between resource efficiency and the effectiveness of detection. %We obtain necessary and sufficient conditions required to characterize the equilibrium solutions of the game between DIFT and APT. 
Our game model is a nonzero-sum, infinite-horizon, average reward stochastic game. Our model incorporates the information asymmetry between players that arises from DIFT's inability to distinguish malicious flows from benign flows and APT's inability to know the locations where DIFT performs a security analysis. Additionally, the game has incomplete information as the transition probabilities (false-positive and false-negative rates) are unknown.
%Additionally, the game has incomplete and imperfect information structure due to unknown transition probabilities and simultaneous actions of the players, respectively. 
We propose a multiple-time scale stochastic approximation algorithm to learn an equilibrium solution of the game. We prove that our algorithm converges to an average reward Nash equilibrium. We evaluated our proposed model and algorithm on a real-world ransomware dataset and validated the effectiveness of the proposed approach.
%\keywords{Security of computer systems \and Advance persistent threats  \and Dynamic information flow tracking \and Stochastic games \and Reinforcement learning.}
%\blfootnote{This work was supported by ONR  grant N00014-16-1-2710 P00002 and DARPA TC grant DARPA FA8650-15-C-7556.} 
\end{abstract}

\IEEEpeerreviewmaketitle
%%%%%%%%%%%%%%%%%%%%%%%%%%%%%%%%%%%%%%%%%%
\section{Introduction}

Advanced Persistent Threats (APTs) are emerging class of cyber threats that victimize governments and organizations around the world through cyber espionage and sensitive information hijacking \cite{jang2014survey, watkins2014impact}. Unlike ordinary cyber threats (e.g., malware, trojans) that execute quick damaging attacks, APTs employ sophisticated and stealthy attack strategies that enable unauthorized operation in the victim system over a prolonged period of time \cite{ussath2016advanced}. End goal of an APT typically aims to sabotage critical infrastructures (e.g., Stuxnet \cite{FalMurChi:11}) or exfiltrate sensitive information (e.g., Operation Aurora, Duqu, Flame, and Red October \cite{BenPekButFel:12}).  APTs follow a multi-stage stealthy attack approach to achieve the goal of the attack. Each stage of an APT is customized to exploit set of vulnerabilities in the victim to achieve a set of sub-goals (e.g., stealing user credentials, network reconnaissance) that will eventually lead to the end goal of the attack \cite{yadav2015technical}. The stealthy, sophisticated and strategic nature of APTs make detecting and mitigating them challenging using conventional security mechanisms such as firewalls, anti-virus softwares, and intrusion detection systems that  heavily rely on the signatures of malware or anomalies observed in the benign behavior of the system.  

Although APTs operate in stealthy manner without inducing any suspicious abrupt changes in the victim system, the interactions of APTs with the system  introduce information flows in the victim system. Information flows consist of data and control commands that dictate how data is propagated between different system entities (e.g., instances of a computer program, files, network sockets) \cite{NewSon-05,ClaLiOrs:07}. Dynamic Information Flow Tracking (DIFT) is a mechanism developed to dynamically track the usage of information flows during program executions \cite{NewSon-05,SuhLeeZhaDev-04}.  Operation of DIFT is based on three core steps. i)~Taint (tag) all the information flows that originate from the set of system entities susceptible to cyber threats \cite{NewSon-05}, \cite{SuhLeeZhaDev-04}. ii)~Propagate the tags into the output information flows based on a set of predefined tag propagation rules which track the mixing of tagged flows with untagged flows at the different system entities. iii)~Verify the authenticity of the tagged flows by performing a security analysis at a subset of system entities using a set of pre-specified tag check rules. When a tagged (suspicious) information flow is verified as malicious through a security check, DIFT makes use of the tags of the malicious information flow to identify victimized system entities of the attack and reset or delete them to protect the system. Since information flows capture the footprint of APTs in the victim system and DIFT allows tracking and inspection of information flows,  DIFT has been recently employed as a  defense mechanism against APTs  \cite{BroTon-16},  \cite{EncGilHanTen-14}. 
%Operation of  DIFT is based on the idea of tainting, tracking, and verifying authenticity of the tainted information flows. 
%Tag propagation rules and tag check rules are defined by system's security experts and often called  {\em security policy} of the DIFT. 

Tagging and tracking information flows in a system using DIFT adds additional resource costs to the underlying system in terms of memory and storage. In addition, inspecting information flows demands extra processing power from the system. Since APTs maintain the characteristics of their malicious information flows (e.g., data rate, spatio-temporal patterns of the control commands) close to the characteristics of benign information flows \cite{wagner2002mimicry} to avoid detection, pre-defined security check rules of DIFT can miss the detection of APTs (false-negatives) or raise false alarms by identifying benign flows as malicious flows (false-positives). Typically, the number of benign information flows exceeds the number of malicious information flows in a system by a large factor. As a consequence, DIFT incurs a tremendous resource and performance overhead to the underlying system due to frequent security checks and false-positives. The high cost and performance degradation of DIFT can be  worse in large scale systems such as servers used in the data centers \cite{EncGilHanTen-14}. 

There has been software-based design approaches to reduce the resource and performance cost of DIFT \cite{SuhLeeZhaDev-04, JiLeeDowWanFazKimOrsLee-17}. However,  widespread deployment of DIFT across various cyber systems and platforms is heavily constrained by the added resource and performance costs that are inherent to DIFT's implementation and due to false-positives and false-negatives generated by DIFT \cite{jee2013shadowreplica},  \cite{nightingale2008parallelizing}.  An analytical model of DIFT need to capture the system level interactions between DIFT and APTs, and cost of resources and performance overhead due to security checks. Additionally, false-positives and false-negatives generated by DIFT also need to be considered while deploying DIFT to detect APTs.

%resource efficient DIFT as an effective defense mechanism that can detect and prevent threats imposed by the APTs.

%A typical APT attack consists of one or more of the following steps \cite{BroTon-16}. APTs initially carry out a reconnaissance on the victim system to identify vulnerabilities and establish the foothold using various techniques such as social engineering and spear-phishing. Then APTs will start to search and gather system information and compromise system resources needed to obtain necessary privilege escalations \cite{BroTon-16} to reach the final stage of the attack. This phase is often called as the {\em lateral movement} of APTs and it consists of multiple stages \cite{BroTon-16}. In the final stage of the attack APTs will use the resources gathered during the lateral movement phase to achieve a particular goal such as 

%Therefore, in this paper we propose a game-theoretic formulation on the underlying system information flow graph to facilitate 
  %However, we find that there is a lack of effort in the systems security literature in this direction.      
In this paper we consider a computer system equipped with DIFT that is susceptible to an attack by APT and provide a game-theoretic model that enables the study of trade-off between resource efficiency and effectiveness of detection of DIFT. Strategic interactions of an APT to achieve the malicious objective while evading detection depends on the effectiveness of the DIFT's defense policy. On the other hand, determining a resource-efficient policy for DIFT that maximize the detection probability depends on the nature of APT's interactions with the system. Non-cooperative game theory provides a rich set of rules that can model  the strategic interactions between two competing agents (DIFT and APT). %The feasible interactions among system processes and objects (e.g., files, network sockets) can be abstracted into a graph called {\em information flow graph} \cite{JiLeeDowWanFazKimOrsLee-17}.  
The contributions of this paper are the following.
%We  model the interaction between DIFT and the system as a noncooperative game. 
%Tracking conditional branching requires analysis of multiple execution traces  \cite{ClaLiOrs:07}.  In order to incorporate the unceratinity involved at the conditional branches, we formulate the interaction between APTs and DIFT as a {\em stochastic game}. 
\begin{enumerate}
	\item[$\bullet$] We model the long-term, stealthy, strategic interactions between DIFT and APT as a two-player, nonzero-sum, average reward, infinite-horizon {\em stochastic  game}. The proposed game model captures the resource costs associated with DIFT in performing security analysis as well as the false-positives and false-negatives of DIFT.  
	\item[$\bullet$] We provide, a reinforcement learning-based algorithm, RL-ARNE, that learns an average reward Nash Equilibrium of the game between DIFT and APT. RL-ARNE is a multiple-time scale algorithm that extends to $K$-player, non-zero sum, average reward, unichain stochastic games.
	%RL-ARNE is a multiple-time scale stochastic approximation algorithm.
	\item[$\bullet$] We prove the convergence of RL-ARNE algorithm to an average reward Nash equilibrium of the game. 
	%The proposed algorithm exploits the structure of the game and is based on the two-time scale  algorithm in \cite{prasad2015two}.  
	%Our algorithm uses the fact that at least  one of the state in the game is reachable from any other state for any pair of strategies of the players.
	\item[$\bullet$] We evaluate the performance of our approach via an experimental analysis on ransomware attack data obtained from Refinable Attack INvestigation (RAIN) \cite{JiLeeDowWanFazKimOrsLee-17}. 
\end{enumerate}

\subsection{Related Work}\label{sec:rel}

Stochastic games introduced by Shapley generalize the Markov decision processes to model the strategic interactions between two or more players that occur in a sequence of stages \cite{shapley1953stochastic}. %A stage of a stochastic game is viewed as a \emph{state} of the game (e.g. position of the players, status of the environment where the game is being played).  Interactions of the players in a state can be viewed as a (poly)matrix game \cite{nash1950equilibrium} whose outcome depends on the outcomes of the (poly)matrix games that occur at the other states of the game. At a decision epoch, each player in a stochastic game make their decisions at a state independently from other players' decisions. Then the state of stochastic game probabilistically transition to (possibly) a another state in response to the decisions made by players. Final outcome (payoff) of a stochastic game is affected by the joint decision made by all players at all the states of the game. 
Dynamic nature of stochastic games enables the modeling of competitive market scenarios in economics \cite{Ami-03}, competition within and between species for resources in evolutionary biology \cite{foster1990stochastic}, resilience of cyber-physical systems in engineering \cite{ZhuBas-11}, and secure networks under adversarial interventions in the field of computer/network science \cite{LyeWin-05}. 
 
 %The environment of security games is the underlying  system and the agents are the defense mechanism and the adversary.
 Study of stochastic games is often focused on finding a set of \emph{ Nash Equilibrium} (NE) \cite{nash1950equilibrium} policies for the players such that no player is able to increase their respective payoffs by unilaterally deviating from their NE policies. The payoffs of a stochastic game are usually evaluated under discounted or limiting average payoff criteria \cite{FilVri-12,sobel1971noncooperative}. Discounted payoff criteria, where future rewards of the players are scaled down by a factor between zero and one, is widely used in analyzing stochastic games as an NE is guaranteed to exist for any discounted stochastic game \cite{mertens2003equilibria}. Limiting average payoff criteria considers the time-average of the rewards received by the players during the game \cite{sobel1971noncooperative}. The existence of an NE under limiting average payoff criteria for a general stochastic game is an open problem.  When an NE exists, value iteration, policy iteration, and linear/nonlinear programming based approaches are proposed in the literature to find an NE \cite{FilVri-12,raghavan1991algorithms}. These approaches, however, require the knowledge of transition structure and the reward structure of the game. Also, these solution approaches are only guaranteed to find an exact NE only in special classes of stochastic games, such as zero-sum stochastic games, where rewards of the players sum up to zero in all the game states \cite{FilVri-12}.

Multi-agent reinforcement learning (MARL) algorithms are proposed in the literature to obtain NE policies of stochastic games when the  transition probabilities of the game and reward structure of the players are unknown. In \cite{bowling2001rational} authors introduced two properties, {\em rationality} and {\em convergence}, that are necessary for a learning agent to learn a discounted NE in MARL setting and proposed a WOLF-policy hill climbing algorithm which is empirically shown to converge to an NE. Q-learning based algorithms are proposed to compute an NE in discounted stochastic games \cite{hu2003nash}  and average reward stochastic games \cite{li2007reinforcement}. Although the convergence of these approaches are guaranteed in the case of zero-sum games,  convergence in nonzero-sum games require more restrictive assumptions on the game, such as existence of an unique NE \cite{hu2003nash}. Recently, a two-time scale algorithm to compute an NE of a nonzero-sum discounted stochastic game was proposed in  \cite{prasad2015two} where authors showed the convergence of algorithm to an exact NE of the game. However, designing reinforcement learning (RL)-based algorithms with provable convergence guarantee for computing average reward Nash-equilibrium in non-zero sum, stochastic games remains an open problem.

%A weaker notion of equilibrium in game, referred to as {\em correlated} equilibrium, was considered in  \cite{greenwald2003correlated}  as the computation of NE is hard. 

%We exploit the structure of the game and propose a modified version of the two-time scale algorithm  in \cite{prasad2015two}.

Various game-theoretic models including  deterministic, stochastic, and limited-information security games have been studied to model the interaction between malicious attackers and defenders of networked systems in \cite{alpcan2010network}. Stochastic games have been  used to analyze security of computer networks in the presence of malicious attackers \cite{AlpBas-06}, \cite{NguAlpBas-09}. 
%In \cite{NguAlpBas-09}, the  existence and uniqueness of an equilibrium in an IDS vs. Attacker game is analyzed. On the other hand,  \cite{AlpBas-06} considered a zero-sum IDS vs. Attacker game and analyzed the game using techniques from Markov Decision Process (MDP) \cite{FilVri-12} and Q-learning \cite{hu2003nash}. Security of a distributed IDS that undergo simultaneous attacks by a number of attackers is modeled as a nonzero-sum stochastic game in \cite{ZhuTemBas-10} and a value-iteration  based algorithm is proposed to compute an $\epsilon$-NE.  
%On the other hand, attacker model of APT  is different from that of IDS as APT activities can blend in with normal user and program activities to blindside IDS. 
%The authors of \cite{ganesan2016dynamic} proposed a RL-based stochastic dynamic programming optimization model for dynamic scheduling of cybersecurity analysts for minimizing the risk. 
In \cite{panfili2018game}, authors modeled an attacker/defender problem as a
multi-agent non-zero sum game and proposed a RL algorithm (friend or foe Q-learning) to solve the game. Adversarial multi-armed bandit and Q-learning algorithms were combined to solve a spatial attacker/defender discounted Stackelberg  game in  \cite{klima2016markov}.  
Game-theoretic frameworks were proposed in the literature to model interaction of APTs with the system through a deceptive APT in \cite{huang2019adaptive} and a mimicry attack in \cite{sayin2018game}.

Our prior works used game theory to model  the interaction of APTs and a DIFT-based detection system \cite{MooShaAllVClaBushWenPoo-18_arx,  MooSahClaLeePoo-18, SahXiaClaLeePoo-18, SahMooAllClaLeePoo-19, moothedath2020StochasticDIFT, misra2019learning,
sahabandu2019stochasticGameSec}.
The game models in  \cite{MooShaAllVClaBushWenPoo-18_arx, MooSahClaLeePoo-18, SahXiaClaLeePoo-18} are  non-stochastic as the notions of false alarms and false-negatives are not considered. Recently, a stochastic model  of DIFT-games was proposed in \cite{SahMooAllClaLeePoo-19, moothedath2020StochasticDIFT} when the transition probabilities of the game are  known.  However, the transition probabilities, which are the rates of generation of false alarms and false negatives at the different system components, are often unknown. In \cite{misra2019learning}, the case of unknown transition probabilities was analyzed  and empirical results to compute approximate equilibrium policies was presented. In the conference version of this paper \cite{sahabandu2019stochasticGameSec} we considered discounted DIFT-game with unknown transition probabilities and proposed a two-time scale RL algorithm that converges to NE of the discounted game.   
%In our framework, we consider APTs that target certain confidential information in the system and conduct operations to breach that information. We model the detection of APTs using DIFT as a dynamic stochastic game. 

%In this paper, we consider the case where the transition probabilities are unknown and the attacker model is also modified to a case where the APT can relaunch an attack even if it fails in a previous attempt. Also, the stochastic game considered in this paper is nonzero-sum and we adopt a reinforcement learning approach. 
%Next, we discuss existing literature on learning based approaches for nonzero-sum stochastic games with unknown transition probabilities.

%%%%%%%%%%%%%%%%%%%%%%%%%%%%%%%%%%%%%%%%%
\subsection{Organization of the Paper}
%The remainder of the paper is organized as follows. 
Section~\ref{sec:DefResults} presents the formal definitions and existing results. Section~\ref{sec:prelim} provides system and defender models. Section~\ref{sec:Game} formulates the stochastic game between DIFT and APT. Section~\ref{sec:Equilibrium} analyzes the necessary and sufficient conditions required to characterize the equilibrium of DIFT-APT game.  Section~\ref{sec:Algorithm} presents a RL based algorithm to compute an equilibrium of DIFT-APT game. Section~\ref{sec:Simulations} provides an experimental study of the proposed algorithm on a real-world attack dataset.  Section~\ref{sec:Conclusions} presents the conclusions.

\section{Formal Definitions and Existing Results}\label{sec:DefResults}
\subsection{Stochastic Games}
A stochastic game $\mathbb{G}$ is defined as a tuple $<K, \mathbb{S}, \A, \mathbb{ P}, r>$, where $K$ denotes the number of players, $\mathbb{S}$ represents the state space, $\A := \A_{1} \times \ldots, \times  \A_{K}$ denotes the action space, $\mathbb{ P}$ designates the transition probability kernel, and $r$ represents the reward functions. Here $\mathbb{S}$  and $\A $ are finite spaces. Let $\A_{k} := \cup_{s \in \mathbb{ S}}\A_{k}(s)$ be the action space of the game corresponding to each player $k \in \{1, \ldots, K\}$, where $\A_{k}(s)$ denotes the set of actions allowed for player $k$ at state $s \in \mathbb{ S}$. Let $\boldsymbol{\pi}_{k}$ be the set of stationary policies corresponding to player $k \in \{1, \ldots, K\}$ in $\mathbb{G}$. Then a policy $\pi_{k} \in \boldsymbol{\pi}_{k}$ is said to be a deterministic stationary policy if  $\pi_{k} \in \{0, 1\}^{|\A_{k}|}$ and said to be a stochastic stationary policy if $\pi_{k} \in [0, 1]^{|\A_{k}|}$. Let $\mathbb{P}(s'|s, a_{1}, \ldots, a_{K})$ be the probability of transitioning from state $s \in  \mathbb{ S}$ to a state $s' \in  \mathbb{ S}$ under set of actions $(a_{1}, \ldots, a_{K})$, where $a_{k} \in \A_{k}(s)$ denotes the action chosen by player $k$ at the state $s$. Further let $r_{k}(s,a_{1}, \ldots, a_{K},s')$ be the reward received by the player $k$ when state of the game transitions from states $s$ to $s'$ under set of actions $(a_{1}, \ldots, a_{K})$ of the players at state $s$. 
%We let the stochastic game to be general sum  $\sum\limits_{k = 1}^{K} r_{k}(s,a_{1}, \ldots, a_{K},s')
%Consider a $K$-player finite general sum stochastic game whose state space is given by $\mathbb{ S}$.  

%Let $k$ be a player in the set $\{D, A\}$ and $-k$ denote the opponents of player $k$.

\subsection{Average Reward Payoff Structure}\label{subsec:ARNE}

Let $\pi = ({\pi}_{1}, \ldots, {\pi}_{K})$. Then define $\rho_{k}(s, \pi)$ to be the average reward payoff of player $k$ when the game starts at an arbitrary state $s \in \mathbb{ S}$ and the players follow their respective policies $\pi$. Let $s^t$  and $a_k^t$ be the state of game at time $t$ and the action of player $k$ at time $t$, respectively. Then $\rho_{k}(s, \pi)$ is defined as 
\begin{equation}\label{eq:average_reward}
\rho_{k}(s, \pi) = \liminf\limits_{T \rightarrow \infty}\frac{1}{T+1}\sum\limits_{t = 0}^{T}\mathbb{E}_{s,\pi}[r_{k}(s^t, a_1^t, \ldots, a_K^t)], 
\end{equation}
where the term $\mathbb{E}_{s,\pi}[r_{k}(s_t, a_1^t, \ldots, a_K^t)]$ denotes the expected reward at time $t$ when the game starts from a state $s$ and the players draw a set of actions $ (a_1^t, \ldots, a_K^t)$ at current state $s^t$ based on their respective policies from $\pi$. All the players in $\mathbb{G}$ aim to maximize their individual payoff values in Eqn.~\eqref{eq:average_reward}.

Let $-k$ be the opponents of a player $k \in \{1, \ldots, K\}$ (i.e., $-k := \{1, \ldots, K\} \backslash k$). Then let ${\pi}_{-k} := \{{\pi}_{1}, \ldots, {\pi}_{K}\} \setminus {\pi}_{k}$ denotes a set of stationary policies of the opponents of player $k$. Equilibrium of $\mathbb{G}$ under average reward criteria is given below.

\begin{defn}[ARNE]\label{def:Nash}
	A set of stationary policies $\pi^{*} = (\pi_{1}^{*}, \ldots, \pi_{K}^{*})$ forms an Average Reward Nash Equilibrium (ARNE) of $\mathbb{G}$ if and only if
	$\rho_{k}(s, \pi_{k}^{*}, \pi_{-k}^{*}) \geq \rho_{k}(s, \pi_{k}, \pi_{-k}^{*}),$ for all $s \in {\bf S}, \pi_{k} \in \boldsymbol{\pi}_{k}$ and $k \in \{1, \ldots, K\} $.
\end{defn}

A policy  $\pi^{*} = (\pi_{1}^{*}, \ldots, \pi_{K}^{*})$ is referred to as an ARNE policy of $\mathbb{G}$. When all the players follow  ARNE policy, no player $k$ is able to increase its payoff value by unilaterally deviating from its respective ARNE policy $\pi_{k}^{*}$.

\subsection{Unichain Stochastic Games}

Define ${\mathbb{P}}(\pi)$ to be the transition probability structure of $\mathbb{G}$ induced by a set of deterministic player policies $\pi$. Note that ${\mathbb{P}}(\pi)$ is a Markov chain formed in the state space $\mathbb{ S}$. Assumption~\ref{assp:ARNE_asmp} presents a condition on ${\mathbb{P}}(\pi)$.
%Then  the following assumption gives a condition that ensures the existence of an ARNE in $\hat{\Gamma}$.

\begin{assumption}\label{assp:ARNE_asmp}
	Induced Markov chain (MC) ${\mathbb{P}}(\pi)$ corresponding to every deterministic stationary policy set $\pi$ contains exactly one recurrent class of states.
\end{assumption}

Assumption~\ref{assp:ARNE_asmp} imposes a structural constraint on the MC induced by \emph{deterministic} stationary policy set.  Here, the single recurrent class need not necessarily contain all $s \in \mathbb{ S}$. There may exist some transient states in ${\mathbb{P}}(\pi)$. Also note that any $\mathbb{G}$ that satisfies Assumption~\ref{assp:ARNE_asmp} can have multiple recurrent classes in ${\mathbb{P}}(\pi)$ under some \emph{stochastic} stationary policy set $\pi$.  Stochastic games that satisfy Assumption~\ref{assp:ARNE_asmp} are referred to as \emph{unichain }stochastic games. In a special case where the recurrent class contains all the states in the state space, $\mathbb{G}$ is referred  as an\emph{ irreducible} stochastic game \cite{FilVri-12}. 

Let ${\mathbb{R}}_{l}$ and ${\mathbb{T}}$ denote a set of states in the $l^{\text{th}}$ recurrent class of  the induced MC ${\mathbb{P}}(\pi)$ for $l \in \{1, \ldots, L\}$, and a set of transient states in ${\mathbb{P}}(\pi)$, respectively,  where $L$ denote the number of recurrent classes. %Further let ${\rho}_{k}(s)$ be the average reward payoff of player $k$ when the initial state is $s \in \mathbb{ S}$.  
%{\color{blue}For ease of representation, we use the notation ${\rho}_{k}(s)$ to represent $ {\rho}_{k}(s, \pi)$}.
Proposition~\ref{prop:Average_reward_results} gives results on the average reward values of the states in each ${\mathbb{R}}_{l}$ and ${\mathbb{T}}$.   

\begin{prop}[\cite{FilVri-12}, Section~3.2]\label{prop:Average_reward_results}
	The following statements are true for any induced MC ${\mathbb{P}}(\pi)$ of $\mathbb{G}$. % For each $s \in \hat{{\bf S}}$,
	\begin{enumerate}
		\item For $l \in \{1, \ldots, L\}$ and for all $s \in {\mathbb{R}}_{l}$, ${\rho}_{k}(s, \pi) = {\rho}^{l}_{k}$, where each ${\rho}^{l}_{k}$ denotes a real-valued constant.
		\item ${\rho}_{k}(s, \pi) = \sum\limits_{l = 1}^{L}q_{l}(s){\rho}^{l}_{k}$, if $s \in {\mathbb{T}}$, where $q_{l}(s)$ is the probability of reaching a state in $l^{\text{th}}$ recurrent class from $s$.
	\end{enumerate}
\end{prop}

$1)$ in Proposition~\ref{prop:Average_reward_results} implies that the average reward payoff of player $k$ takes the same value ${\rho}^{l}_{k}$ for each state in the $l^{\text{th}}$ recurrent class. $2)$ suggests that the average reward payoff of a transient state is a convex combination of the average payoff values corresponding to $L$ recurrent classes ${\rho}^{1}_{k}, \ldots, {\rho}^{L}_{k}$. Proposition~\ref{prop:Average_reward_results} shows that for any $\mathbb{G}$, the average reward payoffs corresponding to each state solely depends on the average reward payoffs of the recurrent classes in ${\mathbb{P}}(\pi)$.

\subsection{ARNE in Unichain Stochastic Games}
Existence of an ARNE for nonzero-sum stochastic games is open. However, the existence of ARNE is shown for some special classes of stochastic games \cite{FilVri-12}. 

%In Section~\ref{sec:Equilibrium}, we show $\Gamma$ satisfies the condition in Assumption~\ref{assp:ARNE_asmp} and use the following result to show the existence of an ARNE in $\Gamma$.
%allows the cases where single recurrent class include only a subset of states in the state space. In such cases the remaining states outside the recurrent class must be transient states reaching the recurrent class. 

\begin{prop}[\cite{sobel1971noncooperative}, Theorem 2]\label{prop:existance_01}
	Consider a stochastic game that satisfies Assumption~\ref{assp:ARNE_asmp}. Then there exists an ARNE for the stochastic game.
\end{prop}

%Next we define USG in Definition~\ref{def:UnichainSG} and present 
%In Section~\ref{sec:Equilibrium}, we show $\Gamma$ belongs to one such class of stochastic games called \textit{Unichain Stochastic Games (USGs)} that is guaranteed to have an ARNE under stationary policies. The formal definition of an USG is given below.
%\begin{defn}[USG]\label{def:UnichainSG}
%	An Unichain Stochastic Game (USG) is a finite non-cooperative stochastic game where every pure strategy of the players induce exactly one recurrent class of states.
%\end{defn}
%The following proposition present the existence result of ARNE in USG. 
%In Section~\ref{sec:Equilibrium}, we show $\Gamma$ is a USG and leverage the equilibrium results presented in Proposition~\ref{prop:existance_01} and Proposition~\ref{prop:Nes&Suf_01} to characterize the ARNE of $\Gamma$.

%Next we present the necessary and sufficient conditions required to characterize the ARNE in USG. Note that the following result is valid for any $K$-player USG with $K \geq 2$. Since our goal is to analyze the two-player stochastic game, $\Gamma$ between $\D$ and $\A$, we present the results in the scope of the two-player game .

Let $\pi_{k} \in \boldsymbol{\pi}_{k}$  be expressed as $\pi_{k} = [\pi_{k}(s)]_{s \in \mathbb{ S}}$, where ${\pi}_{k}(s) = [{\pi}_{k}(s,a_{k})]_{a_{k} \in \A_{k}(s)}$.
%and $a_{k}$ denotes an action of a player $k \in \{1, \ldots, K\}$ at a state $s \in \hat{{\bf S}}$. {\color{red}{definition $p(s'|s,a,\pi)$}}. 
Further let  $\bar{a} := (a_1, \ldots, a_K)$ and $a_{-k} := \bar{a}\backslash a_{k}$. Define $\mathbb{P}(s'|s,a_{k},\pi_{-k}) = \hspace{-2mm} \sum\limits_{a_{-k} \in \A_{-k}(s)}\mathbb{P}(s'|s,\bar{a})\pi_{-k}(s, a_{-k}),$ where $\mathbb{P}(s'|s,\bar{a})$ is the probability of transitioning to a state $s'$ from state $s$ under action set $\bar{a}$. Also let 
$r_{k}(s,a_{k},{\pi}_{-k}) = \hspace{-5mm}\sum\limits_{s' \in \mathbb{S} }\sum\limits_{a_{-k} \in \A_{-k}(s)}\mathbb{P}(s'|s,\bar{a})r_{k}(s,\bar{a},s')\pi_{-k}(s, a_{-k}),$ where $r_{k}(s,\bar{a},s')$ is the reward for player $k$ under action set $\bar{a}$ when a state transitions from $s$ to $s'$.
Then a  necessary and sufficient condition for characterizing an ARNE of a stochastic game that satisfies Assumption~\ref{assp:ARNE_asmp} is given in the following proposition.

\begin{prop}[\cite{sobel1971noncooperative}, Theorem 4]\label{prop:Nes&Suf_01}
	Under Assumption~\ref{assp:ARNE_asmp}, a set of stochastic stationary policies (${\pi}_{1}, \ldots, {\pi}_{K}$) forms an ARNE in $\mathbb{G}$ if and only if  (${\pi}_{1}, \ldots, {\pi}_{K}$) satisfies,
	\begin{subequations}
		\begin{eqnarray}
		&&\hspace{-7mm} {\rho}_{k}(s, \pi) + v_{k}(s) = r_{k}(s,a_{k},{\pi}_{-k}) +\hspace{-1mm} \sum\limits_{s' \in {\mathbb{ S}}}\mathbb{P}(s'|s,a_{k}, {\pi}_{-k})v_{k}(s') +\lambda_{k}^{s,a_{k}} \nonumber\\ &&\hspace{15 mm}\text{for all}~s \in {\mathbb{ S}},~\hspace{-1mm}a_{k} \in \A_{k}(s),~\hspace{-1mm}k \in \{\hspace{-0.5 mm}1, \ldots K\},~\label{subeq:ARNEcon3} \\
		&&\hspace{-7mm}{\rho}_{k}(s, \pi) - \mu_{k}^{s,a_{k}} = \sum\limits_{s' \in \mathbb{ S}}\mathbb{P}(s'|s,a_{k},{\pi}_{-k}){\rho}_{k}(s') \nonumber \\ &&\hspace{15 mm}\text{for all}~s \in {\mathbb{ S}},~\hspace{-1mm}a_{k} \in \A_{k}(s),~\hspace{-1mm}k \in \{\hspace{-0.5 mm}1, \ldots K\},\label{subeq:ARNEcon4} \\
		&&\hspace{-7mm}\sum\limits_{k \in \{1, \ldots, K\}}\sum\limits_{s \in  \mathbb{ S}}\sum\limits_{a_{k} \in \A_{k}(s)}(\lambda_{k}^{s,a_{k}} + \mu_{k}^{s,a_{k}} ){\pi}_{k}(s,a_{k}) = 0.\label{subeq:ARNEcon5} \\
		&&\hspace{-7mm} \lambda_{k}^{s,a_{k}} \geq 0, ~ \mu_{k}^{s,a_{k}} \geq 0, ~ {\pi}_{k}(s,a_{k}) \geq 0 \nonumber\\&&\hspace{15 mm}\text{for all}~s \in {\mathbb{ S}},~\hspace{-1mm}a_{k} \in \A_{k}(s),~\hspace{-1mm}k \in \{\hspace{-0.5 mm}1, \ldots K\},~\ \label{subeq:ARNEcon1} \\
		&&\hspace{-7mm}\sum_{a_{k} \in \A_{k}(s)}\hspace{-3mm}{\pi}_{k}(s,a_{k}) = 1~\text{for all}~s \in {\mathbb{ S}},~\hspace{-1mm}k \in \{\hspace{-0.5 mm}1, \ldots K\}, \label{subeq:ARNEcon2} 
		\end{eqnarray}
	\end{subequations}
	where $v_{k}(s)$ is the ``value" of the game for player $k$ at $s \in \mathbb{ S}$.
\end{prop}

\subsection{Stochastic Approximation Algorithms}\label{subsec:stochastic_Apprx}

Let $h: \mathcal{R}^{m_z} \rightarrow \mathcal{R}^{m_z}$ be a continuous function of a set of parameters $z \in \mathcal{R}^{m_z}$. Then Stochastic Approximation (SA) algorithms solve a set of equations of the form  $h(z) = 0$ based on the noisy measurements of $h(z)$. The classical SA algorithm takes the following form. 

\begin{equation}\label{eq:SA_basic}
z^{n+1} = z^{n} + \delta^{n}_{z}[h(z^{n}) + w_{z}^{n}],~\text{for $n \geq 0$}
\end{equation} 

Here, $n$ denotes the iteration index and $z^{n}$ denote the estimation of $z$ at $n^{\text{th}}$ iteration of the algorithm. The terms $w_{z}^{n}$ and  $\delta^{n}_{z}$ represent the zero mean measurement noise associated with $z^{n}$ and the step-size of the algorithm, respectively. Note that the stationary points of Eqn.~\eqref{eq:SA_basic} coincide with the solutions of $h(z) = 0$ when the noise term $w_{z}^{n}$ is zero. Convergence analysis of SA algorithms requires investigating their associated Ordinary Differential Equations (ODEs). The ODE form of the SA algorithm in Eqn.~\eqref{eq:SA_basic} is given in Eqn.~\eqref{eq:SA_ODE}. 

\begin{equation}\label{eq:SA_ODE}
\dot{z} = h(z) 
\end{equation}

Additionally, the following assumptions on $\delta^{n}_{z}$ are required to guarantee the convergence of an SA algorithm. 
\begin{assumption}\label{assmp:step-size}
	The step-size $\delta^{n}_{z}$ satisfies, $\sum\limits_{n= 0}^{\infty} \delta^{n}_{z} = \infty$ and $\sum\limits_{n= 0}^{\infty} (\delta^{n}_{z})^{2} = 0$.
\end{assumption}
Few examples of $\delta^{n}_{z}$ that satisfy the conditions given in Assumption~\ref{assmp:step-size} are $\delta^{n}_{z} = 1/n$ and $\delta^{n}_{z} = 1/n\log(n)$. A convergence result that holds for a more general class of SA algorithms is given below.

\begin{prop}[]\label{prop:KC_Lemma}
	Consider an SA algorithm in the following form defined over a set of parameters $z \in \mathcal{R}^{m_z}$ and a continuous function  $h: \mathcal{R}^{m_z} \rightarrow \mathcal{R}^{m_z}$.
	\begin{equation}\label{eq:KC_itr}
	z^{n+1} = \Theta({z^{n} + \delta^{n}_{z}[h(z^{n}) + w_{z}^{n} + \kappa^{n}]}),~\text{for $n \geq 0$},
	\end{equation} 
	where $\Theta$ is a projection operator that projects each $z^{n}$ iterates onto a compact and convex set $\Lambda \in \mathcal{R}^{m_z}$ and $\kappa^{n}$ denotes a bounded random sequence. Let the ODE associated with the iterate in Eqn.~\eqref{eq:KC_itr} is given by, 
	\begin{equation}\label{eq:KC_ODE}
	\dot{z} = \bar{\Theta}(h(z)), 
	\end{equation}
	where $\bar{\Theta}(h(z)) = \lim\limits_{\eta \rightarrow 0} \frac{\Theta(z + \eta h(z)) - z}{\eta}$ and $\bar{\Theta}$ denotes a projection operator that restricts the evolution of  ODE in Eqn.~\eqref{eq:KC_ODE} to the set $\Lambda$. Let the nonempty compact set $\boldmath{Z}$ denotes a set of asymptotically stable equilibrium points of  Eqn.~\eqref{eq:KC_ODE}.
	
	Then $z^{n}$ converges almost surely to a point in $\boldmath{Z}$ as $n \rightarrow \infty$ given the following conditions are satisfied,
	\begin{enumerate}
		\item $\delta^{n}_{z}$ satisfies the conditions in Assumption~\ref{assmp:step-size}.
		\item $\lim\limits_{n \rightarrow \infty} \left(\sup\limits_{\bar{n} > n}\left|\sum\limits_{l = n}^{\bar{n}}\delta^{n}_{z}w_{z}^{n}\right|\right)  = 0$ almost surely.
		\vspace*{1mm}
		\item $\lim\limits_{n \rightarrow \infty}\kappa^{n} = 0$ almost surely.
	\end{enumerate}
\end{prop}

Consider a class of SA algorithms that consist of two interdependent iterates that update on two different time scales (i.e., step-sizes of two iterates are different in the order of magnitude). Let $x \in \mathcal{R}^{m_x}$ and $y \in \mathcal{R}^{m_y}$ and $n \geq 0$. Then the iterates given in the following equations portray a format of such two-time scale SA algorithm. 
\begin{eqnarray}
x^{n+1} = x^{n} + \delta_{x}^{n}[f(x^{n},y^{n}) + w_{x}^{n}], \label{eq:Bor1}\\ 
y^{n+1} = y^{n} + \delta_{y}^{n}[g(x^{n},y^{n}) + w_{y}^{n}]. \label{eq:Bor2}
\end{eqnarray}

The following proposition provides a convergence result related to the aforementioned two-time scale SA algorithm.
\begin{prop}[\cite{borkar1997stochastic}, Theorem~2]\label{Prop:Borkar}
	Consider $x^{n}$ and $y^{n}$ iterates given in Eqns.~\eqref{eq:Bor1} and~\eqref{eq:Bor2}, respectively.
	Then, given the iterates in Eqns.~\eqref{eq:Bor1} and~\eqref{eq:Bor2} are bounded, $\{(x^{t},y^{t})\}$ converges to $(\psi(y^{*}),y^{*})$ almost surely under the following conditions.
	\begin{enumerate}
		\item $f:\mathcal{R}^{m_x + m_y} \rightarrow \mathcal{R}^{m_x}$ and $g:\mathcal{R}^{m_x + m_y} \rightarrow \mathcal{R}^{m_y}$ are Lipschitz.\label{con:Bor1}
		\item Iterates $x^{n}$ and $y^{n}$ are bounded. \label{con:Bor2}
		\item Let $\psi: y \rightarrow x$. For all $y \in \mathcal{R}^{m_y}$, the ODE $\dot{x} = f(x,y)$ has an asymptotically stable critical point $\psi(y)$ such that function $\psi$ is Lipschitz. \label{con:Bor3}
		\item The ODE $\dot{y} = g(\psi(y),y)$ has a global asymptotically stable critical point. \label{con:Bor4}
		\item Let $\xi^{n}$ be an increasing $\sigma $-field defined by $\xi^{n} := \sigma({x}^{n}, \ldots, {x}^{0},{y}^{n}, \ldots, {y}^{0},{w}_{x}^{n-1}, \ldots, {w}_{x}^{0}, {w}_{y}^{n-1}, \ldots, {w}_{y}^{0})$. Further let $\kappa_{x}$ and $\kappa_{y}$ be two positive constants. Then $w_{x}^{n}$ and $w_{y}^{n}$ are two noise sequences that satisfy, $\mathbb{E}[ w_{x}^{n}|\xi^{n}] = 0$, $\mathbb{E}[w_{y}^{n}|\xi^{n}] = 0$, $\mathbb{E}[\parallel w_{x}^{n} \parallel^{2}|\xi^{n}] \leq \kappa_{x}(1 + \parallel {x}^{n} \parallel + \parallel {y}^{n} \parallel),$ and  $\mathbb{E}[\parallel w_{y}^{n} \parallel^{2}|\xi^{n}] \leq \kappa_{y}(1 + \parallel {x}^{n} \parallel + \parallel {y}^{n} \parallel)$. \label{con:Bor5}
		\item $\delta_{x}^{n}$ and $\delta_{y}^{n}$ satisfy conditions in Assumption~\ref{assmp:step-size}. Additionally, $\lim\limits_{n \rightarrow \infty} \sup \frac{\delta_{y}^{n}}{\delta_{x}^{n}} = 0$. \label{con:Bor6}
	\end{enumerate}
\end{prop} 

%Assume $w_{x}^{n} >> w_{y}^{n}$ for all $n \geq 0$, then iterate 
%%%%%%%%%%%%%%%%%%%%%%%%%%%%%%%%%%%%%%%%%%
\section{System and Defender Models}\label{sec:prelim}
In this section we detail the concept of information flow graph and the details on the DIFT defender model.

\subsection{Information Flow Graph}
 Information Flow Graph (IFG), $\G=(V_{\G}, E_{\G})$, is a representation of the computer system, where the set of nodes, $V_{\G} = \{u_1, \ldots, u_{N}\}$ depicts the $N$ distinct components of the computer and the set of edges $E_{\G} \subset V_{\G} \times V_{\G}$ represents the feasibility of transferring information flows between the components. Specifically, an edge $e_{ij} \in E_{\G}$ indicates that an information flow can be transferred from a component $u_{i}$ to another component $u_{j}$, where $i,j \in \{1, \ldots, N\}$ and $i \neq j$. Let $\En \subset V_{\G}$ be the set of entry points used by $A$ to infiltrate the computer system. Consider an APT attack that consists of $M$ attack stages and let $\D_{j} \subset V_{\G}$ for each $j \in \{1, \ldots, M\}$ be the set of components that are targeted by the APT in the $j^{\text{th}}$ attack stage. Let $\D_{j}$ be the set of \emph{destinations} of stage $j$.

\subsection{DIFT Defender Model}
DIFT tags/taints all the information flows originating from the set of entry points as suspicious flows. Then DIFT tracks the propagation of the tainted flows through the system and initiates security analysis at specific components of the system to detect the APT. Performing security analysis incurs memory and performance overheads to the system which varies across the system components.
%when an unauthorized use of tainted flow is observed. 
The objective of DIFT  is to 
%detect and mitigate the threats imposed by $A$ via analyzing the authenticity of the tagged flows at a set of components. Also, DIFT need to 
 select a set of system components for performing security analysis while minimizing the memory and performance overhead.
%at which the security costs associated with inspecting the tagged flows in order to minimize performance overhead on the system. 
On the other hand, the objective of APT is to evade detection by DIFT and successfully complete the attack by sequentially reaching at least one node from each set $\D_{j}$, for all $j = 1, \hdots, M$. 

\section{Problem Formulation: DIFT-APT Game}\label{sec:Game}
In this section, we model the interactions between a DIFT-based defender $(D)$ and an APT adversary $(A)$ as a two-player stochastic game (DIFT-APT game). The DIFT-APT game unfolds in the infinite time horizon $t \in \T :=\{1, 2,\ldots \}$.

\subsection{State Space and Action Space}\label{subsec-state_space}
Let ${\bf S}  := \{s_0\} \cup \{V_{\G} \times \{1, \hdots, j\}\}  = \{ s_0, s_1^j, \ldots, s_{N}^j\}$, for all $j \in \{1, \hdots, M\}$, represent the finite state space of DIFT-APT game. The state $s_0$ represents the reconnaissance stage of the attack where APT chooses an entry point of the system to launch the attack. Therefore, at time $t = 0$, DIFT-APT game starts from $s_0$. A state $s_i^j$ denotes a tagged information flow at a system component $u_{i} \in V_{\G}$ corresponding to the $j^{\text{th}}$ attack stage. Also, note that a state $s_i^j$, where $u_{i} \in \D_{j}$ and $j \in \{1, \ldots, M-1\}$, is associated with APT achieving the intermediate goal of stage $j$. Moreover, a state $s_i^M$, where $u_{i} \in \D_{M}$,  represents APT achieving the final goal of the attack. 

%We explain here the construction of the state space graph. Let $\mathcal{N}(s)$ denotes the set of out-neighboring states of state $s \in {\bf S}$. For $s = s_0$, $\mathcal{N}(s) = \{s_{i}^{1} : u_{i} \in \En\}$. For a state $s = s_{i}^{j}$ with $u_{i} \notin \D_{j}$,  $\mathcal{N}(s) = \{s_0\} \cup \{s_{i'}^{j} : \{ (u_{i}, u_{i'}) \in E_{\G}\} $.  

Let $\mathcal{N}(s)$ be the set of out-neighboring states of state $s \in {\bf S}$. Let $\A_{k} = \cup_{s \in {\bf S}}\A_{k}(s)$ be the action space of the player $k \in \{D, A\}$, where $\A_{k}(s)$ denotes the set of actions allowed for player $k$ at a state $s$. The action sets of  the players at any state $s \in {\bf S}$ is given by $\A_{\sD}(s) \in \mathcal{N}(s)\cup \{0\}$ and $\A_{\sA}(s) \in \mathcal{N}(s)\cup \{\varnothing\} $. Here, $\A_{\sD}(s) \in \mathcal{N}(s)$ and $\A_{\sD}(s) = 0$ denote DIFT deciding to perform security analysis at an out-neighboring state and deciding not to perform security analysis, respectively. Also, $\A_{\sA}(s) \in \mathcal{N}(s)$ represents APT deciding to transition to an out-neighboring state of $s \in {\bf S}$ and $\varnothing$ represents APT quitting the attack. At each step of the game DIFT and APT \textit{simultaneously} choose their respective actions. 
%Typically, entry points, $\En$ of an attack consists of network sockets and attack flows are often spoofed to have whitelisted IP addresses. Also, an initial compromise of an attack that occurs at a process leaves no suspicious traces as APTs employ advanced stealthy techniques such as SQL injection. 

%Therefore, there is not enough information to track the authenticity of tagged flows at $\En$. Furthermore, the destinations $\D_{j}$ for $j = \{1, \ldots, M\}$ usually consist of busy process and/or confidential files with restricted access. Hence, performing a security analysis at a destination incurs a high security cost for $D$. Therefore, we assume $D$ is not performing a security analysis at any states corresponding to an entry point or a destination of the attack.

Specifically, there are four cases. (i) $s = s_{0}$, $\A_{\sA}(s) = \{s_{i}^{1} : u_{i} \in \En\}$ and $\A_{\sD}(s) = 0$. Here, APT selects an entry point in the system to initiate the attack.
%Case~(ii): $s = s_{i}^{1}$, where $u_{i} \in \En$, $\A_{\sA}(s) = \{s_{i'}^{1} : (u_{i}, u_{i'}) \in E_{\G}\} \cup \{\varnothing \}$ and $\A_{\sD}(s) \in \mathcal{N}(s)\cup \{0\}$. In other words, APT selects an action to transition to a node $u_{i'}$ in stage one and DIFT chooses to perform security analysis at an out-neighboring state or not. 
(ii)~$\{s = s_{i}^{j}: u_{i} \not\in \D_{j},  j = 1, \hdots, M\}$, $\A_{\sA}(s) = \{s_{i'}^{j} :  (u_{i}, u_{i'}) \in E_{\G}\} \cup \{\varnothing \}$ and $\A_{\sD}(s) \in \{s_{i'}^{j}:  (u_{i}, u_{i'}) \in E_{\G}\} \cup \{0\}$.  In other words,  APT chooses to transition to one of the out-neighboring node of $u_i$ in stage $j$ or decides to quit the attack ($\varnothing$) and DIFT decides to perform security analysis at an out-neighboring node of $u_i$ in stage $j$ or not. (iii) $\{s = s_{i}^{j}: u_{i} \in \D_{j}, j = 1, \hdots, M-1\}$, $\A_{\sA}(s) = s_{i}^{j+1}$ and $\A_{\sD}(s) = 0$. That is, APT traverses from stage $j$ of the attack to stage $j+1$ and DIFT does not perform a security analysis. (iv) $\{s = s_{i}^{M} : s_{i} \in \D_{M}\}$. Then, $\A_{\sA}(s) = s_0$ which captures the persistency of the APT attack  and $\A_{\sD}(s) = 0$.   

Note that  DIFT does not perform security analysis at the states corresponding to $s_0$ and destinations due to the following reasons. At the entry points there are not enough traces to perform security analysis as attack originates at these system components. The destinations $\D_{j}$, for $j \in \{1, \ldots, M\}$, typically consist of busy processes and/or confidential files with restricted access. Performing security analysis at states corresponding to entry points and destinations is not allowed.

%State $s_0$ is a virtual state that captures the reconnaissance stage of APT attack and it does not represent any component of the system.   Lastly, 
%ability of APT to relaunch the attack in the system
\subsection{Policies  and Transition Structure}\label{subsec-transition}
 
Let $s_t$ be the state of the game at time $t \in \T$. Consider \emph{stationary} policies for DIFT and APT, i.e., decisions made at a state $s_t \in \bf S$ at any time $t$ only depends on $s_t$. Let ${\boldsymbol{\pi}}_{\sD}$ and $\boldsymbol{\pi}_{\sA}$ be the set of stationary policies of DIFT and APT, respectively. Then stochastic stationary policies of DIFT and APT are defined by $\pi_{k} \in [0, 1]^{|\A_{k}|}$, where $\pi_{k} \in {\boldsymbol{\pi}}_{k}$ and  $k \in \{D, A\}$. Moreover, let $\pi_{k} = [\pi_{k}(s)]_{s \in {\bf S}}$ and $\pi_{k}(s) = [\pi_{k}(s,a_k)]_{a_k \in \A_{k}(s)}$, where $\pi_{k}(s)$ and $\pi_{k}(s,a_k)$ denote the policy of a player $k \in \{D, A\}$ at a state $s \in {\bf S}$ and probability of player $k$ choosing an action $a_{k} \in \A_{k}(s)$ at the state $s$.  In what follows, we use $a_{k} = d$ when $k = D$ and $a_{k} = a$ when $k = A$ to denote an action of DIFT and APT at a state $s$, respectively.

Assume state transitions are \emph{stationary}, i.e., state at time $t+1$, ${s}_{t+1}$  depends only on the current state ${s}_t$ and the actions $a_t$ and $d_t$ of both players at the state ${s}_t$, for any $t \in \T$. Let $\mathbf{P}$ be the transition structure of the DIFT-APT game. Then $\mathbf{P}(\pi_{\sD},\pi_{\sA})$ represents the state transition matrix of the game resulting from $(\pi_{\sD}, \pi_{\sA}) \in ({\boldsymbol \pi}_{\sD},  {\boldsymbol \pi}_{\sA})$. Then,
$$\mathbf{P}(\pi_{\sD},\pi_{\sA}) = \left[{\bf P}(s'|s,\pi_{\sD},\pi_{\sA})\right]_{s,s' \in {\bf S}},~\mbox{where}$$  
\begin{equation}\label{eq:Pssda}
{\bf P}(s'|s,\pi_{\sD},\pi_{\sA}) \hspace{-0.5mm}= \hspace{-1mm}\sum\limits_{d \in \A_{\sD}(s)}\sum\limits_{a \in \A_{\sA}(s)}{\bf P}(s'|s,d,a)\pi_{\sD}(s,d)\pi_{\sA}(s,a).
\end{equation}

%Moreover, $\mathbf{P}(s'|s, d, a) = 1$ in the following two cases: (i)~when $d=0$ and $a  \in \A_{\sD}(s) \backslash \varnothing$ and (ii)~when  $a = \varnothing$ and $s' = s_0$. Consider a state $s_t=s$ with defender's action $d_t = 1$ and adversary's action  $a_t \neq \varnothing$. Then $0 < \mathbf{P}(s'|s, a, d) = \mathbf{P}(s'|s, a, 1)  < 1$  due to the false negatives associated with DIFT's security analysis.

Here $\mathbf{P}(s'|s, d, a)$ denotes the probability of transitioning to state $s'$ from state $s$ when DIFT chooses an action $d \in \A_{\sD}(s)$ and APT chooses an action $a \in \A_{\sA}(s)$. Let $FN(s_i^j)$ denote the rate of false negatives generated at a system component $u_i \in \V$ while analyzing a tagged flow corresponding to stage~$j$ of the attack. Then for a state $s_{t}$, actions $d_{t}$ and $a_{t}$ the possible next state $s_{t+1}$ are as follows,
\begin{equation}\label{eq:fn_trans}
s_{t+1} =
\begin{cases}
\begin{array}{lll}
s_{i}^{j}, & \mbox{~w.p~} 1, &\mbox{~when~} d_{t} = 0 \mbox{~and~} a_{t} =s_{i}^{j}\\
s_{i}^{j}, & \mbox{~w.p~} FN(s_{i}^j), &\mbox{~when~} d_{t} = a_{t} =s_{i}^{j}\\
s_0, & \mbox{~w.p~} 1-FN(s_{i}^j), &\mbox{~when~} d_{t} = a_{t} =s_{i}^{j} \\
s_{i}^{j}, & \mbox{~w.p~} 1, &\mbox{~when~} d_{t} \neq a_{t} \\
s_0, & \mbox{~w.p~} 1, &\mbox{~when~} a_{t} = \varnothing.
\end{array}
\end{cases}
\end{equation}
In the first case of Eqn.~\eqref{eq:fn_trans}, the next state of the game is uniquely defined by the action of APT as DIFT does not perform security analysis. In the second and thrid cases of Eqn.~\eqref{eq:fn_trans}, DIFT decides correctly to perform security analysis on the malicious flow. Note that the security analysis of DIFT can not accurately detect a possible attack due to generation of false negatives. Hence the next state of the game is determined by the action of APT (in case two) when a false negative is generated. And the next state of the game is $s_0$ (in case three) when APT is detected by DIFT and APT starts a new attack. Case four of  Eqn.~\eqref{eq:fn_trans} represents DIFT performing security analysis on a benign flow. In such a case, the state of the game is uniquely defined by the action of the adversary. Finally, in case five of Eqn.~\eqref{eq:fn_trans}, i.e., when APT decides to quit the attack, the next state of the game is the initial state $s_0$. 

%The state transitions from a state $s = s_i^j$ to state $s' = s_0$ under $a = \varnothing$ or $d=1$ represents APT reattempting the attack after APT quitting the attack and getting detected by DIFT.
%where $u_{i} \notin \D_M$ and $j \neq M$ 

False negatives of the DIFT scheme arise from the limitations of the security rules that can be deployed at each node of the IFG (i.e., processes and objects in the system). Such limitations are due to variations in  the number of rules and the depth of the security analysis\footnote{Detecting an unauthorized use of tagged flow crucially depends on the path traversed by the information flow \cite{ ClaLiOrs:07, SuhLeeZhaDev-04}.} (e.g., system call level trace, CPU instruction level trace) that can be implemented at each node of the IFG resulting from the resource constraints including memory, storage and processing power imposed by the system on each IFG node.

\subsection{Reward Structure}\label{subsec-reward}

Let  ${r}_{\sD}(s,\pi_{\sD}, \pi_{\sA})$ and ${r}_{\sA}(s, \pi_{\sD}, \pi_{\sA})$ be the expected reward of DIFT and APT at a state $s \in {\bf S}$ under policy pair $(\pi_{\sD}, \pi_{\sA}) \in ({\boldsymbol\pi}_{\sD}, {\boldsymbol\pi}_{\sA})$. Then for each $k \in \{D,A\}$, 
\begin{equation*}
r_{k}(s,\pi_{\sD}, \pi_{\sA}) = \hspace{-1mm}\sum\limits_{s' \in {\bf S}}\sum\limits_{\substack{a \in \A_{\sA}(s) \\ d \in \A_{\sD}(s)}} \hspace{-3mm} {\bf P}(s'|s,d,a)\pi_{\sD}(s, d)\pi_{\sA}(s, a) r_{k}(s, d, a, s'),
\end{equation*}
where $r_{k}(s, d, a, s')$ denotes the reward of player $k$ when state transition from $s$ to $s'$ under actions $d \in \AD(s)$ and $a \in \AA(s)$ of DIFT and APT, respectively. Moreover, $r_{\sD}(s,d,a,s')$ and $r_{\sA}(s,d,a,s')$ are defined as follows.
\begin{eqnarray*}
%\end{eqnarray*}
%\begin{eqnarray*}
	r_{\sD}(s, d, a,s') \hspace{-3mm}&=& \hspace{-3mm}\begin{cases}
		\begin{array}{ll}
			\alpha_{\sD}^{j} + \C_{\sD}(s) & \mbox{~if~} d = a,~s' = s_0  \\
			\beta_{\sD}^{j}  & \mbox{~if~}  d = 0,~s' \in \{s^{j}_{i}: u_{i} \in \D_{j}\} \\
			%\beta_{\sD}^{j} + \C_{\sD}(s) & \mbox{~if~} d = 1,~s' = s^{j'}_{i'}: u_{i'} \in \D_{j'} \\
			\sigma_{\sD}^{j} + \C_{\sD}(s) & \mbox{~if~}  d \neq 0,~a = \varnothing \\
			\sigma_{\sD}^{j} & \mbox{~if~}  d = 0,~a = \varnothing \\
			\C_{\sD}(s) & \mbox{~if~}  d \neq a \mbox{~and~} d \neq 0\\
			0 & \mbox{~otherwise~}
		\end{array}
	\end{cases} \\
	r_{\sA}(s, d, a,s')\hspace{-3mm} &=& \hspace{-3mm}\begin{cases}
	\begin{array}{ll}
		\alpha_{\sA}^{j} & \mbox{~if~} d = a,~s' = s_0  \\
		\beta_{\sA}^{j} &\mbox{~if~}  s' \in \{s^{j}_{i}: u_{i} \in \D_{j}\} \\
		\sigma_{\sA}^{j} & \mbox{~if~} a = \varnothing  \\
		0 & \mbox{~otherwise~}
	\end{array}
\end{cases} 
\end{eqnarray*}

The reward structure $r_{\sD}(s, d, a,s')$ captures the cost of false positive generation by assigning a cost $\C_{\sD}(s)$ whenever $d \neq a$ such that $d \neq 0$. Note that, $r_{\sD}(s, d, a, s')$ consists of four components (i)~reward term $\alpha_{\sD}^{j} > 0$ for DIFT detecting the APT in $j^{\text{th}}$ stage (ii)~penalty term $\beta_{\sD}^{j} < 0$ for APT reaching a destination of stage $j$, for $j = 1, \hdots, M$ (iii)~reward $\sigma_{\sD}^{j} > 0$ for APT quitting the attack in $j^{\text{th}}$ stage and (iv)~a security cost $\C_{\sD}(s) < 0$ that captures the memory and storage costs associated with performing  a security checks on a tagged flow at a state $s \in \{s^{j}_{i}: u_{i} \not\in \D_{j} \cup \En\}$. On the other hand $r_{\sA}(s, d, a, s')$ consists of three components (i)~penalty term $\alpha_{\sA}^{j} < 0$ if  APT is detected by DIFT in the $j^{\text{th}}$ stage (ii)~reward term $\beta_{\sA}^{j}>0$ for APT reaching a destination of stage $j$, for $j = 1, \hdots, M$ and (iii)~penalty term $\sigma_{\sA}^j < 0$ for APT quitting the attack in $j^{\text{th}}$ stage. Since it is not necessary that $r_{\sD}(s, d, a, s') = -r_{\sA}(s, d, a, s')$ for all $d \in \AD(s)$, $a \in \AA(s)$ and $s, s' \in {\bf S}$, DIFT-APT game is a \emph{nonzero-sum} game.

\subsection{Information Structure}
Both DIFT and APT are assumed to know the current state, $s_t$  of the game, both action sets $\AD(s_t)$ and $\AA(s_t)$, and payoff structure of the DIFT-APT game. But DIFT is unaware whether a tagged flow at $s_t$ is malicious or not and APT does not know the chances of getting detected at $s_t$. This results in an information {\em asymmetry} between the players. Hence DIFT-APT game is an \textit{imperfect} information game. Furthermore, both players are unaware of the transition structure ${\bf P}$ which depend on the rate of false negatives generated at the different states $s_t$  (Eq.~\eqref{eq:fn_trans}). Consequently, the DIFT-APT game is an \textit{incomplete} information game.

\subsection{Solution Concept: ARNE}\label{sec:sol_concept}
APTs are stealthy attackers whose interactions with the system span over a long period of time. Hence, players $D$ and $A$ must consider the rewards they incur over the long-term time horizon when they decide on their policies $\pi_{D}$ and $\pi_{A}$, respectively.  Therefore, average reward payoff criteria is used to evaluate the outcome of DIFT-APT game for a given policy pair $(\pi_{D}, \pi_{A}) \in (\boldsymbol{\pi}_{D}, \boldsymbol{\pi}_{A})$. 
Note that, the DIFT-APT game originates at $s_0$. Thus the average payoff for player $k \in \{D, A\}$ with policy pair $(\pi_{D}, \pi_{A})$ is defined as follows.
\begin{equation*}
\rho_{k}(s_0, \pi_{D}, \pi_{A}) = \liminf\limits_{T \rightarrow \infty}\frac{1}{T+1}\sum\limits_{t = 0}^{T}\mathbb{E}_{s_0,\pi_{D}, \pi_{A}}[r_{k}(s_t,d_t,a_t)].
\end{equation*}
Moreover, a pair of stationary policies $(\pi_{D}^{*}, \pi_{A}^{*})$ forms an ARNE of DIFT-APT game if and only if
$$\rho_{D}(s, \pi_{D}^{*}, \pi_{A}^{*}) \geq \rho_{D}(s, \pi_{D}, \pi_{A}^{*}),~~~\rho_{A}(s, \pi_{D}^{*}, \pi_{A}^{*}) \geq \rho_{A}(s, \pi_{D}^{*}, \pi_{A})$$ for all $s \in {\bf S}, \pi_{k} \in \boldsymbol{\pi}_{k}$.

%This section gives a set of key existing results from game theory and stochastic approximation theory which will be used to derive the theoretical results presented in this paper. In Section~\ref{subsec:ARNE}, we first introduce the notion of equilibrium considered in this paper to analyze $\Gamma$. Then we present a set of equilibrium results associated with a special class of stochastic games which will be used in Section~\ref{sec:Equilibrium} to derive necessary and sufficient conditions that characterize an equilibrium of $\Gamma$. In Section~\ref{subsec:stochastic_Apprx}, we give a set of key convergence results associated with Stochastic Approximation (SA) algorithms which will be used in Section~\ref{sec:Algorithm} to prove the convergence of the proposed algorithm that computes the equilibrium of $\Gamma$.

%The equilibrium results given in Proposition~\ref{prop:existance_01} to Proposition
%The convergence results of SA schemes given in Proposition to Proposition

\section{Analyzing ARNE of the DIFT-APT Game}\label{sec:Equilibrium}
In this section we first show the existence of ARNE in DIFT-APT game. Then we provide necessary and sufficient conditions required to characterize  an ARNE of DIFT-APT game. Henceforth we assume the following assumption holds for the IFG associated with the DIFT-APT game.

\begin{assumption}\label{assu:1}
	 The IFG is acyclic. 
\end{assumption}

Any IFG with set of cycles can be converted into an acyclic IFG without loosing any causal relationships between the components given in the original IFG. One such dependency preserving conversion is \emph{node versioning} given in \cite{milajerdi2019holmes}. Hence this assumption is not restrictive.
%where a node in the IFG is represented by multiple versions of nodes such that each version of the node will have incoming paths with same set of ancestor process nodes in IFG
Let $\mathbf{P}(\pi_{D}, \pi_{A})$ be the MC induced by a policy pair $(\pi_{D}, \pi_{A})$.  The following theorem presents  properties of DIFT-APT game under Assumption~\ref{assu:1}. %which will be used to establish the equilibrium results.

\begin{theorem}\label{thm:property}
	Let the  DIFT-APT game satisfies Assumption~\ref{assu:1}. Then, the following properties hold.
	\begin{enumerate}
		\item $\mathbf{P}(\pi_{D}, \pi_{A})$ corresponding to any $(\pi_{D}, \pi_{A}) \in (\boldsymbol{\pi}_{D}, \boldsymbol{\pi}_{A})$ consists of a single recurrent class of states (with possibly some transient states reaching the recurrent class).
		\item The recurrent class of $\mathbf{P}(\pi_{D}, \pi_{A})$ includes the state $s_0$.
	\end{enumerate}
\end{theorem}
\begin{proof}
	
	Consider a partitioning of the state space such that ${\bf S} = {\bf S}_{1} \cup {\bf S}_{2}$ and ${\bf S}_{1} \cap {\bf S}_{2} = \varnothing$. Here ${\bf S}_{1}$ denotes the set of states that are reachable\footnote{In a directed graph a state $u$ is said to be reachable from state $v$, if there exists a directed path from $v$ to $u$.} from state $s_0$ and ${\bf S}_{2}$ denotes the set of states that are not reachable from $s_0$. We prove $1)$ and $2)$ by showing that ${\bf S}_{1}$ forms a single recurrent class of $\mathbf{P}(\pi_{D}, \pi_{A})$ and ${\bf S}_{2}$ forms the set of transient states.
	
    We first show that in $\mathbf{P}(\pi_{D}, \pi_{A})$, state $s_0$ is reachable from any arbitrary state $s \in {\bf S} \setminus \{s_0\}$. The proof consists of two steps.
	%\cup_{j = 1}^{M} 
	First consider a state $s = s^{j}_{i}$, such that  $u_{i} \notin \D_{M}$ with $j = M$. In other words, the state $s$ is not a state that is corresponding to a final goal of the attack.  Let $s'$ be an out neighbor of $s$. Then $s'$ satisfies one of the two cases. i)~$s' = s_0$ and ii)~$s' = s^{j'}_{i'} \in {\bf S} \setminus \{s_0\}$. Case~i) happens if the APT decides to dropout from the game or if DIFT successfully detects the APT. Thus in case~i) $s_0$ is reachable from $s$.
	%or $s = s^{M}_{i}$ where $u_i \in \D_{M}$
	
	Case~ii) happens when  DIFT does not detect APT and the APT chooses to move to an out neighboring state $s'$. By recursively applying cases i) and ii) at $s'$, we get $s_0$ is reachable  when case~i) occurs at least once.  What is remaining to show is when only case~ii) occurs. In such a case, transitions from $s'$ will eventually reach a state corresponding to a final goal of the attack,  i.e., $s^{M}_{i}$ with $u_{i} \in \D_{M}$, due to the acyclic nature of the IFG imposed by Assumption~\ref{assu:1}. Note that at $s^{M}_{i}$ with $u_{i} \in \D_{M}$  the only transition possible is to $s_0$. This proves that $s_0$ is reachable from any state $s \in {\bf S} \setminus s_0$. 
	
	%Thus $s_0$ is reachable from any state $s = s^{j}_{i}$, such that  $u_{i} \notin \D_{M}$ with $j = M$.
	
	%Now we show that $s_0$ is reachable from $\hat{s} = s^{j}_{i} \in {\bf S}$, where $u_{i} \in \cup_{j = 1}^{M-1} \{\D_{j}\} $. The unique out neighbor state of $\hat{s}$ is $s'' = s^{j+1}_{i}$. The possible out neighbors $s''$ satisfy cases i) and ii). Then the proof follows from the previous argument. 
	
This along with the definition of ${\bf S}_{1}$ implies that ${\bf S}_{1}$ forms a recurrent class of $\mathbf{P}(\pi_{D}, \pi_{A})$. Also as $s_0$ is reachable from any state in ${\bf S}_{2}$ and by the definition of ${\bf S}_{2}$,  ${\bf S}_{2}$ is the set of transient states. This completes the proof.
\end{proof}

%\footnote{When the single recurrent class in Definition~\ref{def:UnichainSG} contains the entire state space $\S$, the resulting stochastic game is said to be an \textit{irreducible stochastic game}}

Corollary~\ref{cor:exist_Gamma} below presents the existence of an ARNE in DIFT-APT using Theorem~\ref{thm:property}.
\begin{cor}\label{cor:exist_Gamma}
	Let the DIFT-APT game satisfies Assumption~\ref{assu:1}. Then, there exits an ARNE for the DIFT-APT game.
\end{cor}
\begin{proof}
	From the condition~1) in Theorem~\ref{thm:property} the DIFT-APT game has a single recurrent class of states in $\mathbf{P}(\pi_{D}, \pi_{A})$ corresponding to any policy pair $(\pi_{D}, \pi_{A})$. As a result Assumption~\ref{assp:ARNE_asmp} holds for DIFT-APT game. Therefore by Proposition~\ref{prop:existance_01} there exits an ARNE in DIFT-APT game.
	%Given Assumption~\ref{assu:1}, $\mathbf{P}(p_{D}, p_{A})$ include a single recurrent class for any $p_{D}, p_{A}$ by the property~1) in Theorem~\ref{thm:property}. Then Assumption~\ref{assp:ARNE_asmp} is satisfied for two player stochastic game $\Gamma$ since set of pure policies are only a subset of stochastic policies by definition. Then the results follows from Proposition~\ref{prop:existance_01}. 
\end{proof}
%Note that the property~1) in Theorem~\ref{thm:property} satisfies for any stochastic stationary policy pair and where as in assumption 1 the condition need to be satisfied only for pure policies of the players. Then observe that condition 1 in theorem 1. make sure that $\Gamma$ satisfies the condition given in assumption 1 as pure policies are always a subset of stochastic policies. Then the existence of ARNE in $\Gamma$ follows by the  proposition 1.

%Then the following corollary provides existence result of Nash Equilibrium in $\Gamma$.

The corollary below gives a necessary and sufficient condition for characterizing an ARNE of DIFT-APT game. Our algorithm for computing ARNE is based on this condition. 
 
\begin{cor}\label{cor:conditions}
	The following conditions characterizes the ARNE of DIFT-APT game.
	\begin{subequations}
		\begin{equation}\label{eq:GNScon1}
		\rho_{k} + v_{k}(s) \geq  r_{k}(s,a_{k},\pi_{-k}) + \sum\limits_{s' \in {\bf S}}\mathbf{P}(s'|s,a_{k},\pi_{-k})v_{k}(s'),
		\end{equation}
		\begin{equation}\label{eq:GNScon2}
		\begin{split}
		\sum\limits_{k \in \{D, A\}}&\sum\limits_{s \in {\bf S}}\sum\limits_{a_{k} \in \A_{k}(s)}\Big(\rho_{k}  + v_{k}(s) - r_{k}(s,a_{k},\pi_{-k})  \\ &- \sum\limits_{s' \in {\bf S}}\mathbf{P}(s'|s,a_{k},\pi_{-k})v_{k}(s')\Big)\pi_{k}(s,a_{k}) = 0,
		\end{split}
		\end{equation}
		\begin{equation}\label{eq:GNScon3}
		\sum\limits_{a_{k} \in \A_{k}(s)}\pi_{k}(s,a_{k}) = 1, ~~~ \pi_{k}(s,a_{k}) \geq 0,
		\end{equation}
		where $\rho_{k}$ denotes the average reward value of player $k$ independent of initial state of the game.
	\end{subequations}
\end{cor}
\begin{proof}
	By Proposition~\ref{prop:Nes&Suf_01}, ARNE of an unichain stochastic game is characterized by conditions \eqref{subeq:ARNEcon3}-\eqref{subeq:ARNEcon2}. The condition \eqref{subeq:ARNEcon3} reduce to \eqref{eq:GNScon1} by substituting $\lambda_{k}^{s,a_{k}} \geq 0$ from  condition \eqref{subeq:ARNEcon1}. Below is the argument for condition \eqref{subeq:ARNEcon4}. 
	
	From Theorem~\ref{thm:property}, the MC induced by $(\pi_{D}, \pi_{A})$, $\mathbf{P}(\pi_{D}, \pi_{A})$, contains only a single recurrent class. As a consequence, from Proposition~\ref{prop:Average_reward_results}, ${\rho}_{k}(s, \pi) = \rho_{k}$ for all $s \in {\bf S}$ and $k \in \{D, A\}$. 
	Thus condition~\eqref{subeq:ARNEcon4} in Proposition~\ref{prop:Nes&Suf_01} reduces to 
	\begin{eqnarray*}
	{\rho}_{k} - \mu_{k}^{s,a_{k}} = \sum\limits_{s' \in {{{\bf S}}}}\mathbf{P}(s'|s,a_{k},{\pi}_{-k})\rho_{k} 
	= \rho_{k} \sum\limits_{s' \in {{{\bf S}}}}\mathbf{P}(s'|s,a_{k},{\pi}_{-k}) = \rho_{k}
	\end{eqnarray*}
Thus, $\mu_{k}^{s,a_{k}} = 0$. Since ${\rho}_{k}(s, \pi) = \rho_{k}$, condition \eqref{subeq:ARNEcon3} in Proposition~\ref{prop:Nes&Suf_01} becomes
\begin{eqnarray}\label{eq:lambda_sub}
\lambda_{k}^{s,a_{k}}\hspace{-1mm} = \hspace{-0.5mm} \rho_{k}  \hspace{-0.25mm}+ \hspace{-0.25mm}v_{k}(s) \hspace{-0.25mm} - \hspace{-0.25mm}r_{k}(s,a_{k},\pi_{-k}) \hspace{-0.25mm} - \hspace{-1.5mm}\sum\limits_{s' \in {\bf S}}\mathbf{P}(s'|s,a_{k},\pi_{-k})v_{k}(s').
\end{eqnarray}
By substituting $\mu_{k}^{s,a_{k}} = 0$ and $\lambda_{k}^{s,a_{k}}$ from Eqn.~\eqref{eq:lambda_sub},  condition \eqref{subeq:ARNEcon5} reduces to \eqref{eq:GNScon2}. Finally,  conditions \eqref{subeq:ARNEcon1} and \eqref{subeq:ARNEcon2} together reduce to \eqref{eq:GNScon3}. Thus conditions \eqref{eq:GNScon1}-\eqref{eq:GNScon3} characterizes an ARNE in DIFT-APT game.
\end{proof}

%\vspace{-3.5mm}
\section{Design and Analysis of RL-ARNE Algorithm}\label{sec:Algorithm}

In this section we present a RL algorithm that learns ARNE in DIFT-APT game.  %Then we prove the convergence of the algorithm to an ARNE in DIFT-APT game.
\vspace{-3mm}
\subsection{RL-ARNE: Reinforcement Learning Algorithm for Computing Average Reward Nash Equilibrium}
Algorithm~\ref{algo} presents the pseudocode of   RL-ARNE, a stochastic approximation-based algorithm with multiple time scales that computes an ARNE in DIFT-APT game.  The necessary and sufficient condition given in Corollary~\ref{cor:conditions} is used to find an ARNE policy pair $(\pi^\*_{\sD}, \pi^\*_{\sA})$ in Algorithm~\ref{algo}.  
\begin{center}
	\begin{algorithm}[h]
		\caption{RL-ARNE Algorithm of DIFT-APT game}
		\label{algo:Stack}
		\begin{algorithmic}[1]
			\State \textbf{Input:} State space (${\mathbf{S}}$), transition structure  ($\mathbf{P}$), rewards  ($\rD$ and $\rA$), number of iterations ($I >> 0$)
			\State \textbf{Output:} ARNE policies, $(\pi^\*_{\sD}, \pi^\*_{\sA}) \leftarrow (\boldsymbol{\pi}_{\sD}^{I}, \boldsymbol{\pi}_{\sA}^{I})$
			\State \textbf{Initialization: } $n \leftarrow 0$, $v^{0}_{k} \leftarrow 0$, $\rho^{0}_{k} \leftarrow 0$, $\epsilon^{0}_{k} \leftarrow 0$, $\pi^{0}_{k}~\leftarrow~\boldsymbol{\pi}_{k}$ for $k \in \{D, A\}$ and $s \leftarrow s_0$.
			\While  {$n \leqslant I$} 
			\State Draw $d$ from $\pD^{n}(s)$ and $a$ from $\pA^{n}(s)$
			\State Reveal the next state $s'$ according to $\mathbf{P}$
			\State Observe the rewards $r_{D}(s,d,a,s')$ and $r_{A}(s,d,a,s')$
			\For {$k \in \{D, A\}$}
			\State $\hspace*{-5 mm} v^{n+1}_{k}(s) = v^{n}_{k}(s)+\delta_{v}^{n}[r_{k}(s,d,a,s') - \rho^{n}_{k} + v^{n}_{k}(s') - v^{n}_{k}(s)]$ \label{eq:v_itr}
			\State $\hspace*{-5 mm}\rho_{k}^{n+1}  =   \rho_{k}^{n} + \delta_{\rho}^{n}\hspace*{-0.5 mm}\Big[\frac{n\rho^{n}_{k} + r_{k}(s,d,a,s')}{n+1} -    \rho_{k}^{n}\Big]$ \label{eq:rho_itr}
			\State $\hspace*{-5 mm} \epsilon_{k}^{n+1}(s,a_{k}) \hspace*{-0.5 mm} = \hspace*{-0.5 mm}  \epsilon_{k}^{n}(s,a_{k}) + \delta_{\epsilon}^{n} \big[ \sum_{k \in \{D,A\}} (r_{k}(s,d,a,s') - \hspace*{26 mm}  \rho^{n}_{k} + v^{n}_{k}(s')- v^{n}_{k}(s))  - \epsilon_{k}^{n}(s,a_{k}) \big]$ \label{eq:grad_itr}
			\State $\hspace*{-5 mm} \pi_{k}^{n+1}(s,a_{k}) \hspace*{-0.75 mm} = \hspace*{-0.75 mm} \Gamma(\pi_{k}^{n}(s,a_{k}) \hspace*{-0.75 mm} - \delta_{\pi}^{n}\sqrt{\pi_{k}^{n}(s,a_{k})} \big|r_{k}(s,d,a,s')  -\hspace*{26 mm} \rho^{n}_{k} + v^{n}_{k}(s')- v^{n}_{k}(s)\big|\text{sgn}(-\epsilon_{k}^{n}(s,a_{k})))$ \label{eq:policy_itr}
			%\begin{eqnarray}
			%\begin{aligned}
			%\hspace*{-10 mm} v^{n+1}_{k}(s) \hspace*{-3mm} &= &\hspace*{-3mm}  v^{n}_{k}(s)+\delta_{v}^{n}[r_{k}(s,d,a,s') - \rho^{n} + v^{n}_{k}(s') - v^{n}_{k}(s)] %\label{eq:v_itr}\\
			%\hspace*{-10 mm} \rho_{k}^{n+1} \hspace*{-3mm} &=&\hspace*{-3mm}   \rho_{k}^{n} + \delta_{\rho}^{n}\hspace*{-0.5 mm}\Big[\frac{n\rho^{t}_{k} + r_{k}(s,d,a,s')}{n+1} - \rho_{k}^{n}\Big] \label{eq:rho_itr}\\
			%\hspace*{-10 mm}\epsilon_{k}^{n+1}(s,a_{k}) \hspace*{-3mm} &=& \hspace*{-3mm} \epsilon_{k}^{n}(s,a_{k}) + \delta_{\epsilon}^{n}(\hspace*{-3mm} \sum_{k \in \{A,D\}}\hspace*{-4mm} (r_{k}(s,d,a,s') - \rho^{n}_{k} + v^{n}_{k}(s')) \nonumber \\ &&+ \epsilon_{k}^{n}(s,a_{k}) ) \label{eq:grad_itr}\\
			%\hspace*{-10 mm} p_{k}^{n+1}(s,a_{k}) \hspace*{-3mm}&=& \hspace*{-3mm}\Gamma(p_{k}^{n}(s,a_{k}) - \delta_{p}^{n}\sqrt{p_{k}^{n}(s,a_{k})} |r_{k}(s,d,a,s') \nonumber \\&&- \rho^{n}_{k} + v^{n}_{k}(s')|\text{sgn}(\epsilon_{k}^{n}(s,a_{k}))) \label{eq:policy_itr}
			%\end{aligned}
			%\end{eqnarray}
			\EndFor
			\State Update the state of DIFT-APT game: $s \leftarrow s'$
			\State $ n \leftarrow n + 1$
			\EndWhile 
		\end{algorithmic}\label{algo}
	\end{algorithm}
	\vspace*{-3 mm}
\end{center}

Using stochastic approximation, iterates in lines~9  and~10 compute the value functions $v^{n}_{k}(s)$, at each state $s\in {\bf S}$, and average rewards $\rho_{k}^{n}$ of DIFT and APT corresponding to policy pair $(\pi_{\sD}^{n}, \pi_{\sA}^{n})$, respectively.  The iterates, $\epsilon_{k}^{n}(s,a_{k})$ in line~11 and $\pi_{k}^{n}(s,a_{k})$ in line~12, are chosen such that Algorithm~\ref{algo} converges to an ARNE of the DIFT-APT game. We present below the  outline of our approach. 

Let $\Omega_{k, \pi_{-k}}^{s, a_{k}}$ and $\Delta(\pi)$ be defined as
 \begin{eqnarray}\label{eq:Omega}
\hspace{-5mm}\Omega_{k, \pi_{-k}}^{s, a_{k}} \hspace{-4.5mm}&=& \hspace{-3.5mm}\rho_{k} \hspace{-0.5mm}+\hspace{-0.5mm} v_{k}(s) \hspace{-0.5mm}- \hspace{-0.5mm} r_{k}(s,a_{k},\pi_{-k})\hspace{-0.5mm} -\hspace{-2mm} \sum\limits_{s' \in {\bf S}}\hspace{-0.5mm}\mathbf{P}(s'|s,a_{k},\pi_{-k})v_{k}(s') \\
\hspace{-5mm}\Delta(\pi) &=& \sum\limits_{k \in \{D,~A\}}\sum\limits_{s \in {\bf S}}\sum\limits_{a_{k} \in \A_{k}(s)} \Omega_{k, \pi_{-k}}^{s, a_{k}} \pi_{k}(s,a_{k}). \label{eq:Delta}
 \end{eqnarray}
%Recall that conditions \eqref{eq:GNScon1}-\eqref{eq:GNScon3} give a necessary and sufficient condition for ARNE. Using Eqns.~\eqref{eq:Omega} and~\eqref{eq:Delta}, conditions \eqref{eq:GNScon1} and \eqref{eq:GNScon2} can be rewritten as $\Omega_{k, \pi_{-k}}^{s, a_{k}} \geq 0$, for all $a_{k} \in \A_{k}(s)$, $k \in \{D, A\}$, and $s \in {\bf S}$, and $\Delta(\pi) = 0$, respectively.
%in proving the convergence of Algorithm~\ref{algo} (Detailed proof is given in Section~\ref{subsec:convergence}).

 In Theorem~\ref{thm:policy ARNE} we prove that all the policies $(\pi_{D}, \pi_{A})$ such that $\Omega_{k, \pi_{-k}}^{s, a_{k}} < 0$ forms an unstable equilibrium point of the ODE associated with the iterates $\pi_{k}^{n}(s,a_{k})$. Hence, Algorithm~\ref{algo} will not converge to such policies. Consider a policy pair $(\pi_{\sD}, \pi_{\sA})$ such that $\Omega_{k, \pi_{-k}}^{s, a_{k}} \geq 0$.  Note that, by Eqn.~\eqref{eq:Delta}, such a policy pair satisfies $\Delta(\pi) \geq 0$. When $\Delta(\pi) >0$, Algorithm~\ref{algo} updates  the policies of players in a descent direction of $\Delta(\pi)$ to achieve ARNE (i.e., $\Delta(\pi) = 0$). 

Let the gradient of $\Delta(\pi)$ with respect to policies $\pi_D$ and $\pi_A$ be $\frac{\partial\Delta(\pi)}{\partial\pi}$, where $\pi = (\pi_D, \pi_A)$. Then for each $k \in \{D, A\}$, $s \in {\bf S}$, and $a_{k} \in \A_{k}(s)$, $\frac{\partial\Delta(\pi)}{\partial\pi_{k}(s,a_{k})} = \sum\limits_{\bar{k} \in \{D,A\}} \Omega_{\bar{k}, \pi_{-k}}^{s, a_{k}}$ represents each component of $\frac{\partial\Delta(\pi)}{\partial\pi}$. Lemma~\ref{lem:gradient} in Appendix shows the derivation of $\frac{\partial\Delta(\pi)}{\partial\pi_{k}(s,a_{k})}$. Notice that computation of $\frac{\partial\Delta(\pi)}{\partial\pi_{k}(s,a_{k})}$ requires the values of $\mathbf{P}$ which is assumed to be unknown in DIFT-APT game. Therefore the iterate $\epsilon_{k}^{n}(s,a_{k})$ in line~11 of Algorithm~\ref{algo} estimates $\frac{\partial\Delta(\pi)}{\partial\pi_{k}(s,a_{k})}$ using stochastic approximation. Convergence of $-\epsilon_{k}^{n}(s,a_{k})$ to $\frac{\partial\Delta(\pi)}{\partial\pi_{k}(s,a_{k})}$ is proved in Theorem~\ref{thm:gradient_convergence}. 

Additionally, in line~12 of Algorithm~\ref{algo}, the map $\Gamma$ projects the policies to probability simplex defined by  condition~\eqref{eq:GNScon3} in Corollary~\ref{cor:conditions}. Here, $\vert \cdot \vert$ denotes the absolute value. The function $\text{sgn}(\chi)$ denotes the continuous version of the standard sign function (e.g., $\text{sgn}(\chi) = \tanh(c\chi)$ for any constant $c > 1$). Lemma~\ref{lem:valid-descent} shows that the policy iterates in line~12 update  in a valid descent direction of $\Delta(\pi)$ and Theorem~\ref{thm:policy-convergence} proves the convergence. Theorem~\ref{thm:policy ARNE} then shows that the converged policies indeed form an ARNE.

Note that the value function iterates in line~9 and the gradient estimate iterates in line~11 of Algorithm~\ref{algo} update in a same faster time scale $\delta_{v}^{n}$ and  $\delta_{\epsilon}^{n}$, respectively. Policy iterates in line~12 update in a slower time scale $\delta_{\pi}^{n}$. Also average reward payoff iterates in line~10 update in an intermediate time scale $\delta_{\rho}^{n}$. Hence  the step-sizes of the proposed algorithm are chosen such that $\delta_{v}^{n} =  \delta_{\epsilon}^{n}  >> \delta_{\rho}^{n} >> \delta_{\pi}^{n}$. Furthermore, the step-sizes must also satisfy the conditions in Assumption~\ref{assmp:step-size}. Due to time scale separation, iterations in relatively faster time scales see iterations in relatively slower times scales as quasi-static  while the latter sees former as nearly equilibrated \cite{borkar2009stochastic}. 
%We use this fact when defining the ODEs associated with the iterates to study the convergence of the  iterates in 
\vspace{-3mm}

 \begin{remark}
    Note that, RL-ARNE algorithm presented in Algorithm~VI.1 must be trained offline due to the information exchange that is required at line~11 of the algorithm. Here, players are required to exchange the information about their respective temporal difference error estimates, $\tilde{\phi}_{k}(\rho^{n}_{k}, v^{n}_{k}) = r_{k}(s,d,a,s')  - \rho^{n}_{k} + v^{n}_{k}(s')- v^{n}_{k}(s)$, as the iterates on each player's gradient estimation includes the term $\sum_{k \in \{D, A\}}\tilde{\phi}_{k}(\rho^{n}_{k}, v^{n}_{k})$. 
Since RL-ARNE algorithm is trained offline and the policies found at the end of the training only depend on their respective actions, players do not require any information exchange on their respective actions when they execute their learned policies in real-time.
 \end{remark}

\subsection{Convergence Proof of the RL-ARNE Algorithm}\label{subsec:convergence}

First rewrite iterations in line 9 and line 10 as Eqn.~\eqref{eq:v_itr_mod} and Eqn.~\eqref{eq:rho_itr_mod} to show the convergence of value and average reward payoff iterates in Algorithm~\ref{algo}. 
\begin{eqnarray}
v^{n+1}_{k}(s) &=& v^{n}_{k}(s)+\delta_{v}^{n}[F(v^{n}_{k}, \rho^{n}_{k})(s) - v^{n}_{k}(s) + w_{v}^{n}]   \label{eq:v_itr_mod}\\
\rho_{k}^{n+1} &=& \rho_{k}^{n} + \delta_{\rho}^{n} [ G(\rho_{k}^{n})- \rho_{k}^{n} + w_{\rho}^{n}] \label{eq:rho_itr_mod}
\end{eqnarray}
For brevity we use $\pi(s,d,a) = \pi_{D}(s,d)\pi_{A}(s,a)$ and $\pi$ to denote  $(\pi_{D}, \pi_{A})$.
Then, from Eqn.~\eqref{eq:Pssda}, $${\bf P}(s'|s,\pi) = \hspace*{0mm}\sum\limits_{d \in \AD(s)}\sum\limits_{a \in \AA(s) } \pi(s,d,a){\bf P}(s'|s,d,a).$$ Two function maps  $F(v^{n}_{k})(s)$ and $G(\rho_{k}^{n})$ are defined as
\begin{eqnarray}
\hspace*{-5mm}F(v^{n}_{k},\rho^{n}_{k})(s) \hspace*{-2mm}&=& \hspace*{-2mm}\sum\limits_{s' \in {\bf S}}\mathbf{ P}(s'|s,\pi) [ r_{k}(s,d,a,s') - \rho^{n}_{k} + v^{n}_{k}(s')]\label{eq:map_F},\\ 
\hspace*{-5mm}G(\rho_{k}^{n})  \hspace*{-2mm}&=&\hspace*{-2mm} \sum\limits_{s' \in {\bf S}}\mathbf{ P}(s'|s,\pi) \Big[ \frac{n\rho^{n}_{k} + r_{k}(s,d,a,s')}{n+1} \Big].\label{eq:map_G}
\end{eqnarray}

The zero mean noise parameters $w_{v}^{n}$ and $w_{\rho}^{n}$ are defined as
\begin{eqnarray}
w_{v}^{n} &=& r_{k}(s,d,a,s') - \rho^{n}_{k} + v^{n}_{k}(s') - F(v^{n}_{k},\rho^{n}_{k})(s) \label{eq:noise_v},\\
w_{\rho}^{n} &=& \frac{n\rho^{n}_{k} + r_{k}(s,d,a,s')}{n+1}  - G(\rho_{k}^{n}) \label{eq:noise_rho}.
\end{eqnarray}

Let ${v}_{k}  = [v_{k}(s)]_{s \in {\bf S}}$. Then the ODE associated with the iterates given in Eqn.~\eqref{eq:v_itr_mod} corresponding to all $s \in {\bf S}$ and the ODE associated with the iterate in Eqn.~\eqref{eq:rho_itr_mod} are as follows. 
\begin{eqnarray}
\dot{v}_{k} &=& f(v_{k},\rho_{k}) \label{eq:v_itr_ODE}\\
\dot{\rho}_{k} &=& g(\rho_{k}) \label{eq:rho_itr_ODE},
\end{eqnarray}
where $f: \mathcal{R}^{|{\bf S}|} \rightarrow \mathcal{R}^{|{\bf S}|}$ is such that $f(v_{k},\rho_{k}) = F(v_{k},\rho_{k}) - v_{k}$, where $F(v_{k},\rho_{k}) =[F(v_{k},\rho_{k})(s)]_{s \in {\bf S}}$ and $g: \mathcal{R} \rightarrow \mathcal{R}$ is defined as $g(\rho_{k}) = G(\rho_{k}) - \rho_{k}$.

We note that, in Algorithm~VI.1,  value function iterates ($v_k^{n}(s)$) runs in a relatively faster time scale compared to the average reward iterates ($\rho_k^n$). As a consequence, $v_k^{n}(s)$ iterates see $\rho_k^n$ as quasi-static. Hence, for brevity, in the proofs of Lemma~VI.2, Lemma~VI.5, and Theorem~VI.7 we represent $f(v_{k},\rho_{k})$ and $F(v^{n}_{k}, \rho^{n}_{k})(s)$ as $f(v_{k})$ and $F(v^{n}_{k})(s)$, respectively.  

A set of lemmas that are used to prove the convergence of the iterates in lines~9 and~10 of Algorithm~\ref{algo} are given below.
Lemma~\ref{lem:Lipschitz} presents a property of the ODEs in Eqns.~\eqref{eq:v_itr_ODE} and \eqref{eq:rho_itr_ODE}.

\begin{lemma}\label{lem:Lipschitz}
	Consider the ODEs $\dot{v}_{k} = f(v_{k},\rho_{k}) $ and $\dot{\rho}_{k} = g(\rho_{k})$. Then the functions $f(v_{k},\rho_{k})$ and $g(\rho_{k})$ are Lipschitz.
\end{lemma}
 \begin{proof}
 	First we show $f(v_{k})$ is Lipschitz. Consider two distinct value vectors $v_{k}$ and $\bar{v}_{k}$. Then, 
 	\begin{eqnarray}
 	\hspace{-4.5mm}\parallel f(v_{k}) - f(\bar{v}_{k}){\parallel}_{1} \hspace{-3mm}&=& \hspace{-3mm}\parallel [ F(v_{k}) - F(\bar{v}_{k})] - [v_{k}- \bar{v}_{k}] {\parallel}_{1} \nonumber\\
 	\hspace{-4.5mm}\hspace{-3mm}&\leq& \hspace{-3mm}\parallel F(v_{k}) - F(\bar{v}_{k}) {\parallel}_{1} + \parallel v_{k}- \bar{v}_{k} {\parallel}_{1} \nonumber\\
 	\hspace{-4.5mm}\hspace{-3mm}&=& \hspace{-3mm}\hspace{-1mm}\sum\limits_{s \in {\bf S}}\Big| F(v_{k})(s) - F(\bar{v}_{k})(s) \Big| + \hspace{-1mm}\parallel v_{k}- \bar{v}_{k} {\parallel}_{1}. \label{eq:v_lip}
 	\end{eqnarray}
Notice that,
\begin{equation*}
\begin{split}
\sum\limits_{s \in {\bf S}}\Big| F(v_{k})(s) - F(\bar{v}_{k})(s) \Big| &=  \sum\limits_{s \in {\bf S}}\left|\sum\limits_{s' \in {\bf S}}{\bf P}(s'|s,\pi)[v_{k}(s') - \bar{v}_{k}(s')]\right|  \\
 &\leq \sum\limits_{s \in {\bf S}}\sum\limits_{s' \in {\bf S}} {\bf P}(s'|s,\pi)\left|v_{k}(s')-\bar{v}_{k}(s')\right|  
\end{split}
\end{equation*} 
\begin{equation*}
\begin{split}
 &\leq \sum\limits_{s \in {\bf S}}\sum\limits_{s' \in {\bf S}} \left|v_{k}(s')-\bar{v}_{k}(s')\right|  \\
 &= \sum\limits_{s \in {\bf S}}\parallel v_{k}-\bar{v}_{k}{\parallel}_{1} = |{\bf S}|\parallel v_{k}-\bar{v}_{k}{\parallel}_{1}.
\end{split}
\end{equation*} 
The inequalities in the above equations are followed by the triangle inequality and observing the fact that $\max\{{\bf P}(s'|s,\pi)\} =~1$.  Then from Eqn.~\eqref{eq:v_lip}, 
\begin{equation*}
\parallel f(v_{k}) - f(\bar{v}_{k}){\parallel}_{1} \leq (|{\bf S}| + 1)\parallel v_{k}-\bar{v}_{k}{\parallel}_{1}.
\end{equation*}
Hence $f(v_{k})$ is Lipschitz. Next we prove $g(\rho_{k})$ is Lipschitz. Let $\rho_{k}$ and $\bar{\rho}_{k}$ be two distinct average payoff values. Then, 
 \begin{equation*}
 \begin{split}
 \left| g(\rho_{k}) - g(\bar{\rho}_{k}) \right|  &= \left|\frac{n}{n+1} [ \rho_{k}- \bar{\rho}_{k}] - [\rho_{k}- \bar{\rho}_{k}] \right| = \left| \rho_{k}- \bar{\rho}_{k} \right|.
 \end{split}
 \end{equation*}
 Therefore  $g(\rho_{k})$ is Lipschitz.
 \end{proof}

Lemma~\ref{lem:v_pcont} shows the map $F(v^{n}_{k}) = [F(v^{n}_{k})(s)]_{s \in {\bf S}}$ is a pseudo-contraction with respect to some weighted sup-norm. The definitions of weighted sup-norm and pseudo-contraction  are given below.
\begin{defn}[Weighted sup-norm]\label{def:WSN}
	Let $|| b ||_{\epsilon}$ denote the weighted sup-norm of a vector $b \in \mathcal{R}^{m_b}$ with respect to the vecor $\epsilon \in \mathcal{R}^{m_b}$. Then, 
	\begin{equation*}
	|| b ||_{\epsilon} = \max_{q = 1, \ldots, n} \frac{|b(q)|}{\epsilon(q)},
	\end{equation*}
	where $|b(q)|$ represent the absolute value of the $q^{\text{th}}$ entry of vector $b$.
\end{defn}
\begin{defn}[Pseudo contraction]
	Let $c, \bar{c} \in \mathcal{R}^{m_c}$. Then a function $\phi: \mathcal{R}^{m_c} \rightarrow \mathcal{R}^{m_c} $ is said to be a pseudo contraction with respect to the vector $\gamma \in \mathcal{R}^{m_c}$ if and only if, 
	$$\parallel \phi(c) - \phi(\bar{c}) \parallel_{\gamma} \leq \eta \parallel c - \bar{c} \parallel_{\gamma},~\mbox{where}~ 0 \leq \eta < 1 .$$
	
\end{defn}
\begin{lemma}\label{lem:v_pcont}
Consider $F(v^{n}_{k},\rho^{n}_{k})(s)$ defined in Eqn.~\eqref{eq:map_F}. Then the function map $F(v^{n}_{k},\rho^{n}_{k}) = [F(v^{n}_{k},\rho^{n}_{k})(s)]_{s \in {\bf S}}$ is a pseudo-contraction with respect to some weighted sup-norm.
\end{lemma}
\begin{proof}
	Consider two distinct value functions $v^{n}_{k}$ and $\bar{v}^{n}_{k}$. Then, 
\begin{eqnarray}
	\hspace*{-3mm}&&\hspace*{-3mm}\parallel F(v^{n}_{k})(s) - F(\bar{v}^{n}_{k})(s){\parallel}_{1}  = \parallel\sum_{s' \in {\bf S}} {\bf P}(s'|s,\pi) [ v_{k}^{n}(s')- \bar{v}_{k}^{n}(s')] {\parallel}_{1} \nonumber\\ 
	\hspace*{-3mm} &=&\hspace*{-3mm} \parallel\sum\limits_{s' \in {\bf S}} \sum\limits_{\substack{d \in \AD(s)\\ a \in \AA(s) }}\pi(s,d,a){\bf P}(s'|s,d,a) [ v_{k}^{n}(s')- \bar{v}_{k}^{n}(s')] {\parallel}_{1} \nonumber\\
	 \hspace*{-3mm} &\leq& \sum\limits_{\substack{d \in \AD(s)\\ a \in \AA(s) }}\pi(s,d,a)\sum_{s' \in {\bf S}} {\bf P}(s'|s,d,a)  \parallel v_{k}^{n}(s')- \bar{v}_{k}^{n}(s'){\parallel}_{1} \label{eq:v_pcont}
	\end{eqnarray}
	Eqn.~\eqref{eq:v_pcont} follows from triangle inequality. To find an upper bound for the term ${\bf P}(s'|s,d,a)$ in Eqn.~\eqref{eq:v_pcont}, we construct a Stochastic Shortest Path Problem (SSPP) with the same state space and transition probability structure as in DIFT-APT game, and a player whose action set is given by $\AD \times \AA$. Further set the rewards corresponding to all the state transition in SSPP to be $-1$. Then by Proposition~2.2 in \cite{bertsekas1996neuro}, the following holds condition for all $s \in {\bf S}$ and $(d,a) \in \AD(s) \times \AA(s)$.
	\begin{equation*}
	\sum_{s' \in {\bf S}} {\bf P}(s'|s,d,a)\epsilon(s') \leq \eta\epsilon(s),
	\end{equation*}
	where $\epsilon \in [0,1]^{|{\bf S}|}$ and $ 0 \leq \eta < 1$. Rewrite Eqn.~\eqref{eq:v_pcont} as
	\begin{equation*}
	\begin{split}
	&| F(v^{n}_{k})(s) - F(\bar{v}^{n}_{k})(s)| 	\end{split}
	\end{equation*}
	\begin{equation*}
	\begin{split}	
	&\leq \sum\limits_{\substack{d \in \AD(s)\\ a \in \AA(s) }}\pi(s,d,a)\sum_{s' \in {\bf S}} {\bf P}(s'|s,d,a)\epsilon(s')  \frac{| v_{k}^{n}(s')- \bar{v}_{k}^{n}(s') |}{\epsilon(s')}\\
	&\leq \sum\limits_{\substack{d \in \AD(s)\\ a \in \AA(s) }}\pi(s,d,a)\sum_{s' \in {\bf S}} {\bf P}(s'|s,d,a)\epsilon(s')  \max\limits_{s' \in {\bf S}}\frac{| v_{k}^{n}(s')- \bar{v}_{k}^{n}(s') |}{\epsilon(s')}\\
	&\leq \sum\limits_{\substack{d \in \AD(s)\\ a \in \AA(s) }}\pi(s,d,a)\sum_{s' \in {\bf S}} {\bf P}(s'|s,d,a)\epsilon(s')  \parallel  v_{k}^{n}- \bar{v}_{k}^{n} {\parallel}_{\epsilon}\\
	&\leq \sum\limits_{\substack{d \in \AD(s)\\ a \in \AA(s) }}\pi(s,d,a)\eta{\epsilon(s)} \parallel  v_{k}^{n}- \bar{v}_{k}^{n} {\parallel}_{\epsilon} = \eta{\epsilon(s)} \parallel  v_{k}^{n}- \bar{v}_{k}^{n} {\parallel}_{\epsilon}.
	\end{split}
	\end{equation*}
	\begin{equation*}
	\begin{split}
	\frac{| F(v^{n}_{k})(s) - F(\bar{v}^{n}_{k})(s)|}{\epsilon(s)}  &\leq  \eta\parallel  v_{k}^{n}- \bar{v}_{k}^{n} {\parallel}_{\epsilon} \\
	\max\limits_{s\in{\bf S}}\frac{| F(v^{n}_{k})(s) - F(\bar{v}^{n}_{k})(s)|}{\epsilon(s)}  &\leq  \eta\parallel  v_{k}^{n}- \bar{v}_{k}^{n} {\parallel}_{\epsilon} \\
	\parallel F(v^{n}_{k}) - F(\bar{v}^{n}_{k}){\parallel}_{\epsilon}  &\leq  \eta\parallel  v_{k}^{n}- \bar{v}_{k}^{n} {\parallel}_{\epsilon}.
	\end{split}
	\end{equation*}
\end{proof}
The next result proves the boundedness of the iterates in Algorithm~\ref{algo}.
\begin{lemma}\label{v_rho_bound}
	Consider the RL-ARNE algorithm presented in Algorithm~\ref{algo}. Then, the iterates $v^{n}_{k}(s)$ and $\rho^{n}_{k}$, for $s\in {\bf S}$ and $k \in \{D, A\}$, in Eqn.\eqref{eq:v_itr_mod} and Eqn.\eqref{eq:rho_itr_mod} are bounded.
\end{lemma}
\begin{proof}
Lemma~\ref{lem:v_pcont} proved that $F(v^{n}_{k})$ is a pseudo-contraction with respect to some weighted sup-norm.  By choosing step-size, $\delta_{v}^{n}$ to satisfy Assumption~\ref{assmp:step-size} and observing that the noise parameter, $w_{v}^{n}$ is zero mean with bounded variance, all the conditions in Theorem~1 in \cite{tsitsiklis1994asynchronous} hold for the DIFT-APT game. Hence, by Theorem~1 in \cite{tsitsiklis1994asynchronous}, the iterates $v^{n}_{k}(s)$ in Eqn.~\eqref{eq:v_itr_mod} are bounded for all $s \in {\bf S}$. 

From Proposition~\ref{prop:Average_reward_results}, a fixed policy pair $(\pi_{D}, \pi_{A})$ and $n >> 0$, the average reward payoff values $\rho^{n}_{k}$  depend only on the rewards due to the state transitions that occur within the recurrent classes of induced MC.  Recall that Theorem~\ref{thm:property} showed induced Markov chain,${\bf P}(\pi_{D}, \pi_{A})$, in DIFT-APT game contains only a single recurrent class. Let ${\bf S}_{1}$ be the set of states in the recurrent class of ${\bf P}(\pi_{D}, \pi_{A})$. Then there exists a unique stationary distribution $p$ for ${\bf P}(\pi_{D}, \pi_{A})$ restricted to states in ${\bf S}_{1}$. Thus for $n >> 0$ and each $k \in \{D, A\}$, 
\begin{equation}\label{eq:rho_critical}
\rho^{n}_{k} = \sum\limits_{s \in {\bf S}_{1}} p(s)r_{k}(s, \pi),
\end{equation}
 where $p(s)$ is the probability of being at state $s \in {\bf S}_{1}$ and $r_{k}(s, \pi) = \sum\limits_{\substack{d \in \AD(s)\\ a \in \AA(s) }}\pi(s,d,a)\sum_{s' \in {\bf S}} {\bf P}(s'|s,d,a)r(s,d,a,s')$ is the expected reward at the state $s \in {\bf S}_{1}$ for player $k \in \{D, A\}$. Since ${\bf S_{1}}$ has finite carnality and the rewards, $r_{k}$ are finite for DIFT-APT game, $\rho^{n}_{k}$ converge to a globally asymptotically stable critical point given in Eqn.~\eqref{eq:rho_critical} and $\rho^{n}_{k}$ iterates are bounded. 
 %This completes the proof.
%Therefore, $\rho^{n}_{k}$  iterates are bounded since the state space is finite and the rewards are bounded in $\Gamma$. 
\end{proof}
Theorem~\ref{thm:convergence} proves the convergence of the iterates $v^{n}_{k}(s)$, for all $s\in {\bf S}$, and $\rho^{n}_{k}$.
\begin{theorem}\label{thm:convergence}
	Consider the RL-ARNE algorithm presented in Algorithm~\ref{algo}. Then the iterates $v^{n}_{k}(s)$, for all $s\in {\bf S}$, and $\rho^{n}_{k}$ for $k \in \{D, A\}$ in Eqn.~\eqref{eq:v_itr_mod} and Eqn.~\eqref{eq:rho_itr_mod} converge.
\end{theorem}
\begin{proof}
	By Proposition~\ref{Prop:Borkar}, convergence of the stochastic approximation-based algorithm by conditions \eqref{con:Bor1}-\eqref{con:Bor6}. Lemma~\ref{lem:Lipschitz} and Lemma~\ref{v_rho_bound} showed that condition~\eqref{con:Bor1} and condition~\eqref{con:Bor2}  in Proposition~\ref{Prop:Borkar} are satisfied, respectively. 
	
	To show that condition~\eqref{con:Bor3} is satisfied, we first show $F(v^{n}_{k})(s)$ is a non-expansive map. Consider two distinct value functions $v^{n}_{k}$ and $\bar{v}^{n}_{k}$. Since $P(s'|s,\pD,\pA) \leq 1$, from Eqn.~\eqref{eq:v_pcont},
	\begin{equation*}
	\parallel F(v^{n}_{k})(s) - F(\bar{v}^{n}_{k})(s){\parallel}  \leq \parallel  v_{k}^{n}(s')- \bar{v}_{k}^{n}(s') {\parallel}.
	\end{equation*}
	Thus $F(v^{n}_{k})(s)$ is a non-expansive map and hence from Theorem 2.2 in \cite{soumyanath1999analog} iterates $v^{n}_{k}(s)$, for all $s\in {\bf S}$ and $k \in \{D, A\}$, converge to an asymptotically stable critical point. Thus condition~\eqref{con:Bor3} is satisfied. Lemma~\ref{v_rho_bound}, showed that $\rho^{n}_{k}$, for $k \in \{D, A\}$, converge to a globally asymptotically stable critical point which implies that condition~\eqref{con:Bor4} is satisfied. 
	
	From Eqns.~\eqref{eq:noise_v} and~\eqref{eq:noise_rho}, the noise measures have zero mean. The variance of these noise measures are bounded by the fineness of the rewards in  DIFT-APT game and the boundedness of the iterates $v^{n}_{k}(s)$ and $\rho^{n}_{k}$. Thus condition~\eqref{con:Bor5} is satisfied. Finally, the choice of step-sizes to satisfy condition~\eqref{con:Bor6}. Therefore the results follows by Proposition~\ref{Prop:Borkar}.
	%Notice that iterates $\rho^{n}_{k}$ tracks the average reward payoff of $\Gamma$ for each player in the game. Since the state space and action space of $\Gamma$ is finite and all the reward parameters, $r_{k}(s,d,a,s')$ takes bounded values, given any policy pair $(\pD, \pA)$ induce a single recurrent class of states in $\Gamma$, $\rho^{n}_{k}$ converge to a globally asymptotically stable critical point.
\end{proof}

Next theorem proves the convergence of gradient estimates.

\begin{theorem}\label{thm:gradient_convergence}
	Consider $\Omega_{k, \pi_{-k}}^{s, a_{k}}$ and $\Delta(\pi)$ given in Eqns.~\eqref{eq:Omega} and~\eqref{eq:Delta}, respectively. Then gradient estimation iterate, $\epsilon_{k}^{n}(s,a_{k})$ in line~11 corresponding to any $k \in \{D, A\}$, $s \in {\bf S}$, and $a_{k} \in \A_{k}(s)$, converge to $-\frac{\partial\Delta(\pi)}{\partial\pi_{k}(s,a_{k})} = -\sum\limits_{\bar{k} \in \{A,D\}} \Omega_{\bar{k}, \pi_{-k}}^{s, a_{k}}$. 
\end{theorem}
\begin{proof}
	Rewrite gradient estimation in line~11 as follows. 
	\begin{equation}\label{eq:gradient_extended}
	\epsilon_{k}^{n+1}(s,a_{k})  =   \epsilon_{k}^{n}(s,a_{k}) + \delta_{\epsilon}^{n} \big[  -\hspace{-3.25mm}\sum\limits_{\bar{k} \in \{D,A\}}\hspace{-2mm} \Omega_{\bar{k}, \pi_{-k}}^{s, a_{k}} \hspace{-0.5mm} - \epsilon_{k}^{n}(s,a_{k}) + w^{n}_{\epsilon} \big],
	\end{equation}
	where $w^{n}_{\epsilon} \hspace{-0.2mm}= \hspace{-0.5mm}\sum\limits_{k\in \{D,A\}} \bar{\Omega}_{k}^{s} +\hspace{-0.3mm} \sum\limits_{\bar{k} \in \{D,A\}} \Omega_{\bar{k}, \pi_{-k}}^{s, a_{k}}$, and $\bar{\Omega}_{k}^{s} = r_{k}(s,d,a,s') -\rho^{n}_{k} + v^{n}_{k}(s')- v^{n}_{k}(s)$. Note that $\mathbb{ E}(w^{n}_{\epsilon}) = 0$. Then ODE associated with Eqn.~\eqref{eq:gradient_extended} is given by, 
	\begin{equation*}
	\dot{\epsilon}_{k}(s,a_{k})= -\sum\limits_{\bar{k} \in \{D,A\}}\hspace{-1.5mm} \Omega_{\bar{k}, \pi_{-k}}^{s, a_{k}}  - \epsilon_{k}(s,a_{k}).
	\end{equation*}
	
	We use Proposition~\ref{prop:KC_Lemma} to prove the convergence of  gradient estimation iterates,  $\epsilon_{k}^{n}(s,a_{k})$. Step-size $\delta_{\epsilon}^{n}$ is chosen such that condition 1) in Proposition~\ref{prop:KC_Lemma} is satisfied. Validity of condition 2) can be shown as follows.
	
	\begin{equation}\label{eq:doob}
	\mathbb{ E}\left(\lim\limits_{n \rightarrow \infty} \left(\sup\limits_{\bar{n} > n}\left|\sum\limits_{l = n}^{\bar{n}}\delta^{l}_{\epsilon}w_{\epsilon}^{l}\right|^{2}\right)\right)  \leq 4\lim\limits_{n \rightarrow \infty}\sum\limits_{l = n}^{ \infty}(\delta^{l}_{\epsilon})^{2}\mathbb{ E}(|w_{\epsilon}^{l}|^{2}) = 0.
	\end{equation}
	 Inequality in Eqn.~\eqref{eq:doob} follows by Doob inequality \cite{metivier1984applications}. Equality in Eqn.~\eqref{eq:doob} follows by choosing $\delta_{\epsilon}^{n}$ to satisfy Assumption~\ref{assmp:step-size} and observing $\mathbb{ E} (|w^{l}_{\epsilon}|^{2}) < \infty$ as $r_{k}$, $v^{n}_{k}$, and $\rho^{n}_{k}$ are bounded in DIFT-APT game. Comparing Eqn.~\eqref{eq:gradient_extended} with Eqn.~\eqref{eq:KC_itr},  $\kappa = 0$ in Eqn.~\eqref{eq:gradient_extended}. Therefore, from Proposition~\ref{prop:KC_Lemma}, as $n \rightarrow \infty$,  $\epsilon_{k}^{n}(s,a_{k})  \rightarrow -\sum\limits_{\bar{k} \in \{A,D\}} \Omega_{\bar{k}, \pi_{-k}}^{s, a_{k}} = -\frac{\partial\Delta(\pi)}{\partial\pi_{k}(s,a_{k})}$. This completes the proof showing the convergence of gradient estimation iterates $\epsilon_{k}^{n}(s,a_{k})$.
\end{proof}

Next, we prove the convergence of the policy iterates. In order to do so, we proceed in the following manner.
\begin{enumerate}
    \item We rewrite the conditions in Corollary~\ref{cor:conditions} that characterize ARNE of DIFT-APT game as a non-linear optimization problem (Problem~\ref{prob:ARNE-NLP}).
    \item Then we show the policies are updated in a valid decent direction, $\sqrt{\pi_{k}^{n}(s,a_{k})} \big|\Omega_{k, \pi_{-k}}^{s, a_{k}}\big|\text{sgn}\left(\frac{\partial\Delta(\pi^{n})}{\partial\pi_{k}^{n}(s,a_{k})}\right)$, with respect to the objective function (or temporal difference error), $\Delta(\pi)$, of Problem~\ref{prob:ARNE-NLP} (Lemma~\ref{lem:valid-descent}).
    \item Using steps~1) and~2), we characterize the stable and unstable equilibrium points associated with the ODE corresponding to the policy iterates in line~12 (Lemma~\ref{lem:stabel-Equi}). 
    \item Invoking Proposition~\ref{prop:KC_Lemma} we prove the convergence of policy iterates to stable equilibrium points found in step~3) (Theorem~\ref{thm:policy-convergence}).
\end{enumerate}
Below we elaborate steps~1)-4). ARNE of the DIFT-APT game can be characterized as the following non-linear optimization problem (step~1)).

\begin{prob}\label{prob:ARNE-NLP}  The necessary and sufficient conditions given in Corollary~\ref{cor:conditions} that characterize the ARNE of DIFT-APT can be reformulated as the following non-linear program using  $\Omega_{k, \pi_{-k}}^{s, a_{k}}$ and $\Delta(\pi)$ introduced in Eqns.~\eqref{eq:Omega} and~\eqref{eq:Delta}.
\begin{equation*}
\begin{aligned}
\min_{v,\rho,\pi} \quad \Delta(\pi) 
\textrm{~s.t.} \quad  \Omega_{k, \pi_{-k}}^{s, a_{k}} \geq 0; \hspace{-2mm}
   \sum\limits_{a_{k} \in \A_{k}(s)}\pi_{k}(s,a_{k}) = 1;
   ~\pi_{k}(s,a_{k}) \geq 0,    
\end{aligned}
\end{equation*}   
where $v = (v_{D}, v_{A}),~v_{k} = [v_{k}(s)]_{s \in \mathbf{S}}~\text{for}~k \in \{D, A\},~\rho = (\rho_{D}, \rho_{A}),~\pi = (\pi_D, \pi_A),~\pi_{k} = [\pi_{k}(s)]_{s \in \mathbf{S}}$, and $\pi_{k}(s) = [\pi_{k}(s, a_k)]_{a_k \in \A_{k}(s)}$, for $k \in \{D, A\}$.
\end{prob}

In Lemma~VI.10, we show policy iterates are updated in a valid descent direction with respect to the objective function, $\Delta(\pi)$ (step~2)).

\begin{lemma}\label{lem:valid-descent}
    Consider $\Omega_{k, \pi_{-k}}^{s, a_{k}}$ and $\Delta(\pi)$ given in Eqns.~\eqref{eq:Omega} and~\eqref{eq:Delta}, respectively. Then for any $k \in \{D, A\}$, $s \in {\bf S}$, and $a_{k} \in \A_{k}(s)$, policy iterate, $\pi_{k}^{n}(s,a_{k})$, in line~12 of Algorithm~\ref{algo} is updated in a valid descent direction, $\sqrt{\pi_{k}^{n}(s,a_{k})} \big|\Omega_{k, \pi_{-k}}^{s, a_{k}}\big|\text{sgn}\left(\frac{\partial\Delta(\pi^{n})}{\partial\pi_{k}^{n}(s,a_{k})}\right)$, of $\Delta(\pi)$ when $\Omega_{k, \pi_{-k}}^{s, a_{k}} \geq 0$ and $\Delta(\pi) > 0$. 
\end{lemma}
\begin{proof}
First we rewrite policy iteration in line~12 as follows. 
\begin{equation}
\begin{split} \label{eq:extended_policy}
\pi_{k}^{n+1}(s,a_{k}) \hspace*{-0.75 mm} = \hspace*{-0.75 mm} \Gamma(\pi_{k}^{n}(s,a_{k}) \hspace*{-0.75 mm} - \delta_{\pi}^{n}\left(\sqrt{\pi_{k}^{n}(s,a_{k})} \big|\Omega_{k, \pi_{-k}}^{s, a_{k}}\big| \right. \\ \left.\text{sgn}\left(\frac{\partial\Delta(\pi^{n})}{\partial\pi_{k}^{n}(s,a_{k})}\right)+ w^{n}_{\pi}\right) ), 
\end{split}
\end{equation}
where $ w^{n}_{\pi} = \sqrt{\pi_{k}^{n}(s,a_{k})} \left[ \big|\bar{\Omega}_{k}^{s} \big|- \big| \Omega_{k,\pi_{-k}}^{s,a_{k}}\big|\right]\text{sgn}\left(\frac{\partial\Delta(\pi^{n})}{\partial\pi_{k}^{n}(s,a_{k})}\right)$, and $\bar{\Omega}_{k}^{s} = r_{k}(s,d,a,s') -\rho_{k} + v_{k}(s')- v^{n}_{k}(s)$. Note that, policy iterate updates in the slowest time scale when compared to the other iterates. Thus, Eqn.~\eqref{eq:extended_policy} uses the converged values of value functions $(v_{k})$, average reward values $(\rho_{k})$, and gradient estimates $(\frac{\partial\Delta(\pi^{n})}{\partial\pi_{k}^{n}(s,a_{k})})$ with respect to policy $\pi^{n} = (\pi_{\sD}^{n}, \pi_{\sA}^{n})$.

Consider a policy $\pi_{k}^{n+1}$ whose entries are same as $\pi_{k}^{n}$ except the entry $\pi_{k}^{n+1}(s,a_{k})$ which is chosen as in Eqn.~\eqref{eq:extended_policy}, for small $ 0 <\delta_{\pi}^{n} << 1$. Let $\bar{\pi} = (\pi_{k}^{n+1}, \pi_{-k}^{n})$ and $\hat{\pi} = (\pi_{k}^{n}, \pi_{-k}^{n})$. Also note that $\mathbb{ E}(w^{n}_{\pi}) = 0$. Thus ignoring the term $w^{n}_{\pi}$ and using Taylor series expansion yields, 
\begin{equation*}
\begin{split}
\Delta(\bar{\pi}) =& \Delta(\hat{\pi}) + \delta_{\pi}^{n}\left(-\sqrt{\pi_{k}^{n}(s,a_{k})} \big|\Omega_{k, \pi_{-k}}^{s, a_{k}}\big| \right. \\ &\left. \text{sgn}\left(\frac{\partial\Delta(\hat{\pi})}{\partial\pi_{k}^{n}(s,a_{k})}\right)\frac{\partial\Delta(\hat{\pi})}{\partial\pi_{k}^{n}(s,a_{k})}\right) + o(\delta_{\pi}^{n})\\
= &\Delta(\hat{\pi}) + \delta_{\pi}^{n}\left(-\sqrt{\pi_{k}^{n}(s,a_{k})} \big|\Omega_{k, \pi_{-k}}^{s, a_{k}}\big| \Big|\frac{\partial\Delta(\hat{\pi})}{\partial\pi_{k}^{n}(s,a_{k})}\Big|\right),
\end{split}
\end{equation*}
where $o(\delta_{\pi}^{n})$ represents the higher order terms corresponding to $\delta_{\pi}^{n}$. We ignore $o(\delta_{\pi}^{n})$ in the second equality above since the choice of $\delta_{\pi}^{n}$ is small. Notice that the term $\delta_{\pi}^{n}\left(-\sqrt{\pi_{k}^{n}(s,a_{k})} \big|\Omega_{k, \pi_{-k}}^{s, a_{k}}\big| \Big|\frac{\partial\Delta(\hat{\pi})}{\partial\pi_{k}^{n}(s,a_{k})}\Big|\right)$ is negative. Since $\Delta(\pi) > 0$ for any $\pi$, we get $\Delta(\bar{\pi}) < \Delta(\hat{\pi})$. This proves policies are updated in a valid descent direction. 
\end{proof}

Notice that the ODE associated with Eqn.~\eqref{eq:extended_policy} can be written as, 
	\begin{equation}\label{eq:ODE_policy}
	\dot{\pi}_{k}(s,a_{k}) \hspace*{-0.75 mm} = \hspace*{-0.75 mm} \bar{\Gamma}\left(-\sqrt{\pi_{k}(s,a_{k})} \big|\Omega_{k, \pi_{-k}}^{s, a_{k}}\big| \text{sgn}\left(\frac{\partial\Delta(\pi)}{\partial\pi_{k}(s,a_{k})}\right)\right),
	\end{equation}
	where $\bar{\Gamma}$ is the continuous version of the projection operator $\Gamma$ which is defined analogous to the continuous projection operator in  Eqn.~\eqref{eq:KC_ODE}.
%We introduce the following notations before presenting the next result. 
Let $\Pi$ denotes the set of limit points associated with the system of ODEs in Eqn.~\eqref{eq:ODE_policy}.
Let the feasible set of Problem~\ref{prob:ARNE-NLP} be 
\begin{equation}\label{eq:feasible}
    H = \{\pi \in L \vert \Omega_{k, \pi_{-k}}^{s, a_{k}} \geq 0,~\text{for all}~a_k \in \A_{k}(s),~s \in \mathbf{S},~k \in \{D, A\}\},
\end{equation}
where the set $L = \{\pi \vert \sum_{a_{k} \in \A_{k}(s)}\pi_{k}(s,a_{k}) = 1, \pi_{k}(s,a_{k})  \geq  0, \text{ for all } a_k \in \A_{k}(s),~s \in \mathbf{S}\}$. The set $\Pi$ can be partitioned using the set $H$ as $\Pi = \Pi_1 \cup \Pi_2$, where $\Pi_1 = \Pi \cap H$ and $\Pi_2 = \Pi \setminus \Pi_1$. Using these notations and  steps~1) and~2), we characterize the stable and unstable equilibrium points of the system of ODEs in Eqn.~\eqref{eq:ODE_policy} in Lemma~VI.11 (step~3)).

\begin{lemma}\label{lem:stabel-Equi}
The following statements are true for the set of equilibrium policies $\pi^{\star}$ of ODE in Eqn.~\eqref{eq:ODE_policy}.
\begin{enumerate}
    \item All $\pi^{\star} \in \Pi_1$ form a set of stable equilibrium points.
    \item All $\pi^{\star} \in \Pi_2$ form a set of unstable equilibrium points.
\end{enumerate}
\end{lemma}
\begin{proof}
First we show statement~1) holds. Since the set $\Pi_1$ is in the feasible set $H$ of Problem~\ref{prob:ARNE-NLP} defined in Eqn.~\eqref{eq:feasible}, for any $\pi^{\star} \in \Pi_1$, there exists some $a_k \in \A_{k}(s),~s \in \mathbf{S}$ that satisfy $\Omega_{k, \pi_{-k}}^{s, a_{k}} \geq 0$. Let $B_{\zeta}(\pi^{\star}) = \{\pi \in L \vert \norm{\pi - \pi^{\star}} < \zeta\}$. Then, for any $\pi \in B_{\zeta}(\pi^{\star}) \setminus \Pi_1$, there exists a $\zeta > 0$ such that $\Omega_{k, \pi_{-k}}^{s, a_{k}} > 0$ which yields $\frac{\partial\Delta(\pi)}{\partial\pi_{k}(s,a_{k})} > 0$. This implies $\text{sgn}\left(\frac{\partial\Delta(\pi)}{\partial\pi_{k}(s,a_{k})}\right) > 0$.

Hence, $\bar{\Gamma}\left(-\sqrt{\pi_{k}(s,a_{k})} \big|\Omega_{k, \pi_{-k}}^{s, a_{k}}\big| \text{sgn}\left(\frac{\partial\Delta(\pi)}{\partial\pi_{k}(s,a_{k})}\right)\right) < 0$ for any $\pi \in B_{\zeta}(\pi^{\star}) \setminus \Pi_1$. This implies that ${\pi}_{k}(s,a_{k})$ will decrease when moving away from $\pi^{\star} \in \Pi_1$. This proves $\pi^{\star} \in \Pi_1$ is an stable equilibrium point of the system of ODEs given in Eqn.~\eqref{eq:ODE_policy}.

To show statement~2) is true, we first note that for any $\pi^{\star} \in \Pi_2$, there exists some $a_k \in \A_{k}(s),~s \in \mathbf{S}$ such that $\Omega_{k, \pi_{-k}}^{s, a_{k}} < 0$. Then, for any $\pi \in B_{\zeta}(\pi^{\star}) \setminus \Pi_2$, there exists a $\zeta > 0$ such that $\Omega_{k, \pi_{-k}}^{s, a_{k}} < 0$ which yields $\frac{\partial\Delta(\pi)}{\partial\pi_{k}(s,a_{k})} < 0$. This implies $\text{sgn}\left(\frac{\partial\Delta(\pi)}{\partial\pi_{k}(s,a_{k})}\right) < 0$.
 
 Therefore, $\bar{\Gamma}\left(-\sqrt{\pi_{k}(s,a_{k})} \big|\Omega_{k, \pi_{-k}}^{s, a_{k}}\big| \text{sgn}\left(\frac{\partial\Delta(\pi)}{\partial\pi_{k}(s,a_{k})}\right)\right) > 0$ for any $\pi \in B_{\zeta}(\pi^{\star}) \setminus \Pi_2$. This implies that ${\pi}_{k}(s,a_{k})$ will increase when moving away from $\pi^{\star} \in \Pi_2$. This proves $\pi^{\star} \in \Pi_2$ is an unstable equilibrium point of the system of ODEs in Eqn.~\eqref{eq:ODE_policy} and completes the proof. 
\end{proof}

Theorem~VI.12 gives the convergence 
%property 
of the policy iterates to the set of stable equilibrium points 
%found 
in step 3) (step~4)).

\begin{theorem}\label{thm:policy-convergence}
%Consider $\Omega_{k, \pi_{-k}}^{s, a_{k}}$ and $\Delta(\pi)$ defined in Eqns.~\eqref{eq:Omega} and~\eqref{eq:Delta}. 
Consider the RL-ARNE algorithm presented in Algorithm~\ref{algo}. Then the policy iterates $\pi_{k}^{n}(s,a_{k})$ for all $a_k \in \A_{k}(s)$, $s \in \mathbf{S}$, and $k \in \{D,~ A\}$ converge to a stable equilibrium point $\pi^{\star} = (\pi_{\sD}^{\*}, \pi_{\sA}^{\*}) \in \Pi_1$.
\end{theorem}
\begin{proof}
Recall $ w^{n}_{\pi} = \sqrt{\pi_{k}^{n}(s,a_{k})} \left[ \big|\bar{\Omega}_{k}^{s} \big|- \big| \Omega_{k,\pi_{-k}}^{s,a_{k}}\big|\right]\text{sgn}\left(\frac{\partial\Delta(\pi^{n})}{\partial\pi_{k}^{n}(s,a_{k})}\right)$.
%, and $\bar{\Omega}_{k}^{s} = r_{k}(s,d,a,s') -\rho_{k} + v_{k}(s')- v^{n}_{k}(s)$.
%The proof follows from the similar arguments as in the proof of Theorem~\ref{thm:gradient_convergence} by choosing $\delta_{\pi}^{n}$ to satisfy Assumption~\ref{assmp:step-size} and observing $\mathbb{ E} (|w^{n}_{\pi}|^{2}) < \infty$.
We invoke Proposition~\ref{prop:KC_Lemma} to prove the convergence of  policy iterates,  $\pi_{k}^{n}(s,a_{k})$. Step-size $\delta_{\pi}^{n}$ is chosen such that condition 1) in Proposition~\ref{prop:KC_Lemma} is satisfied. Validity of condition 2) can be shown as follows.
	
	\begin{equation}\label{eq:doob2}
	\mathbb{ E}\left(\lim\limits_{n \rightarrow \infty} \left(\sup\limits_{\bar{n} > n}\left|\sum\limits_{l = n}^{\bar{n}}\delta^{l}_{\pi}w_{\pi}^{l}\right|^{2}\right)\right)  \leq 4\lim\limits_{n \rightarrow \infty}\sum\limits_{l = n}^{ \infty}(\delta^{l}_{\pi})^{2}\mathbb{ E}(|w_{\pi}^{l}|^{2}) = 0.
	\end{equation}
	 Inequality in Eqn.~\eqref{eq:doob2} follows by Doob inequality \cite{metivier1984applications}. Equality in Eqn.~\eqref{eq:doob2} follows by choosing $\delta_{\pi}^{n}$ to satisfy Assumption~\ref{assmp:step-size} and observing $\mathbb{ E} (|w^{l}_{\pi}|^{2}) < \infty$ as $r_{k}$, $v^{n}_{k}$, and $\rho^{n}_{k}$ are bounded in DIFT-APT game. Comparing Eqn.~\eqref{eq:extended_policy} with Eqn.~\eqref{eq:KC_itr}, $\kappa = 0$. Therefore, from Proposition~\ref{prop:KC_Lemma}, as $n \rightarrow \infty$,  the policy iterates $\pi_{k}^{n}(s,a_{k})$ for all $a_k \in \A_{k}(s)$, $s \in \mathbf{S}$, and $k \in \{D,~ A\}$ converge to a stable equilibrium point $\pi^{\star} \in \Pi_1$. This completes the proof showing the convergence of policy iterates $\pi_{k}^{n}(s,a_{k})$.
\end{proof}

Next theorem proves the convergence of $\pi_{k}^{n}(s,a_{k})$ given in line~12 of Algorithm~\ref{algo}, to an ARNE in DIFT-APT game. 

\begin{theorem}\label{thm:policy ARNE}
 Consider $\Omega_{k, \pi_{-k}}^{s, a_{k}}$ and $\Delta(\pi)$ given in Eqns.~\eqref{eq:Omega} and~\eqref{eq:Delta}, respectively. A converged policy $(\pi_{\sD}^{\*}, \pi_{\sA}^{\*})$ of RL-ARNE algorithm presented in Algorithm~\ref{algo}  forms an ARNE in DIFT-APT game.
\end{theorem}
\begin{proof}

In the following, we show any converged policy $\pi^{\star} = (\pi_{\sD}^{\*}, \pi_{\sA}^{\*})$ returned by RL-ARNE algorithm presented in Algorithm~\ref{algo} will satisfy conditions~\eqref{eq:GNScon1}-\eqref{eq:GNScon3} in Corollary~\ref{cor:conditions} and thus $\pi^{\star}$ forms an ARNE in DIFT-APT game.

Recall from Theorem~\ref{thm:policy-convergence}, the policy iterates $\pi_{k}^{n}(s,a_{k})$ for all $a_k \in \A_{k}(s)$, $s \in \mathbf{S}$, and $k \in \{D,~ A\}$ converge to a stable equilibrium point $\pi^{\star} \in \Pi_1$. Also, recall $\Pi$ denotes the set of limit points associated with the system of ODEs in Eqn.~\eqref{eq:ODE_policy} and $L = \{\pi \vert \sum_{a_{k} \in \A_{k}(s)}\pi_{k}(s,a_{k}) = 1, \pi_{k}(s,a_{k})  \geq  0, \text{ for all } a_k \in \A_{k}(s),~s \in \mathbf{S}\}$. Then, from the definition of the set $\Pi_1$, any converged $\pi^{\star}$ will satisfy conditions~\eqref{eq:GNScon1} and~\eqref{eq:GNScon3}, since $\pi^{\star} \in \Pi_1 = \Pi \cap H$ yields $\pi^{\star} \in H$, where $H = \{\pi \in L \vert \Omega_{k, \pi_{-k}}^{s, a_{k}} \geq 0,~\text{for all}~a_k \in \A_{k}(s),~s \in \mathbf{S},~k \in \{D, A\}\}$.

Then it suffices to show any $\pi^{\star} \in \Pi_1$ will yield  $\sqrt{\pi_{k}(s,a_{k})}\Omega_{k, \pi_{-k}}^{s, a_{k}} = 0$ since this proves condition~\eqref{eq:GNScon3} in Corollary~\ref{cor:conditions}. We show this by contradiction arguments. 

Note that $\bar{\Gamma}\left(-\sqrt{\pi_{k}(s,a_{k})} \big|\Omega_{k, \pi_{-k}}^{s, a_{k}}\big| \text{sgn}\left(\frac{\partial\Delta(\pi)}{\partial\pi_{k}(s,a_{k})}\right)\right) = 0$ as $\pi^{\star}$ forms a set of equilibrium polices associated with the system of ODEs in Eqn.~\eqref{eq:ODE_policy}. Then suppose there exists a policy $0 < \pi_{k}(s,a_{k}) \leq 1$ for some $\bar{a}_k \in \A_{k}(s)$, $s \in \mathbf{S}$, and $k \in \{D,~ A\}$ such that $\sqrt{\pi_{k}(s,\bar{a}_{k})}\Omega_{k, \pi_{-k}}^{s, a_{k}} \neq 0$.

Now consider the following two cases.

\noindent {\bf Case~I: $\pi_{k}(s,\bar{a}_{k}) = 1$ and $\Omega_{k, \pi_{-k}}^{s, \bar{a}_{k}} \neq 0$}. 

Recall $F(v_{k},\rho_{k}) =[F(v_{k},\rho_{k})(s)]_{s \in {\bf S}}$ and $ F(v_{k},\rho_{k})(s) = \sum\limits_{s' \in {\bf S}}\mathbf{ P}(s'|s,\pi) [ r_{k}(s,d,a,s') - \rho^{n}_{k} + v_{k}(s')] $.
Then under Case~I, we obtain the following:
$$\sum_{a_k \in \A_{k}(s)} \pi_{k}(s,a_{k})\Omega_{k, \pi_{-k}}^{s, a_{k}}  = \pi_{k}(s,\bar{a}_{k})\Omega_{k, \pi_{-k}}^{s, \bar{a}_{k}} = 0,$$
where the first equality is due to $\pi_{k}(s,\bar{a}_{k}) = 0$ and the second equality is due to the convergence of the value iterates to their true values (i.e., as $n \rightarrow  \infty$, $v_k \rightarrow F(v_{k},\rho_{k})$) which is proved in Theorem~\ref{thm:convergence}.

Further, as $\pi_{k}(s,\bar{a}_{k}) = 1$ this yields $\Omega_{k, \pi_{-k}}^{s, \bar{a}_{k}} = 0$, which contradicts the condition $\Omega_{k, \pi_{-k}}^{s, \bar{a}_{k}} \neq 0$ in Case~I.

\noindent {\bf Case~II: $0 < \pi_{k}(s,\bar{a}_{k}) < 1$ and $\Omega_{k, \pi_{-k}}^{s, \bar{a}_{k}} \neq 0$}. 

Under this case we get
\begin{eqnarray*}
\bar{\Gamma}\left(-\sqrt{\pi_{k}(s,a_{k})} \big|\Omega_{k, \pi_{-k}}^{s, a_{k}}\big| \text{sgn}\left(\frac{\partial\Delta(\pi)}{\partial\pi_{k}(s,a_{k})}\right)\right)  \\ = -\sqrt{\pi_{k}(s,a_{k})} \big|\Omega_{k, \pi_{-k}}^{s, a_{k}}\big| \text{sgn}\left(\frac{\partial\Delta(\pi)}{\partial\pi_{k}(s,a_{k})}\right) \neq 0,
\end{eqnarray*} due to conditions in given in the Case~II and assuming\footnote{This can be achieved by repeating an action in Algorithm~\ref{algo} when $\text{sgn}(\cdot) = 0$. A similar approach has been proposed in the algorithm that computes an NE of discounted stochastic games in \cite{prasad2015two}.} $\text{sgn}(\cdot) \neq 0$. However this contradicts with our initial observation of $\bar{\Gamma}\left(-\sqrt{\pi_{k}(s,a_{k})} \big|\Omega_{k, \pi_{-k}}^{s, a_{k}}\big| \text{sgn}\left(\frac{\partial\Delta(\pi)}{\partial\pi_{k}(s,a_{k})}\right)\right) = 0$.

Therefore, by contradiction, there does not exist any policy $0 < \pi_{k}(s,a_{k}) \leq 1$ for some $\bar{a}_k \in \A_{k}(s)$, $s \in \mathbf{S}$, and $k \in \{D,~ A\}$ such that $\sqrt{\pi_{k}(s,\bar{a}_{k})}\Omega_{k, \pi_{-k}}^{s, a_{k}} \neq 0$.
This proves condition~\eqref{eq:GNScon3} in Corollary~\ref{cor:conditions} holds. 

Since now we have shown conditions~\eqref{eq:GNScon1}-\eqref{eq:GNScon3} in Corollary~\ref{cor:conditions} hold, a converged policy $(\pi_{\sD}^{\*}, \pi_{\sA}^{\*})$ of RL-ARNE algorithm presented in Algorithm~\ref{algo}  forms an ARNE in DIFT-APT game.
\end{proof}

%%%%%%%%%%%%%%%%%%%%%%%%%%%%%%%%%%%%%%%%%%%%%%%%%%%%%%
%{\color{blue}
\begin{remark}
Note that RL-ARNE algorithm presented in Algorithm~VI and the associated convergence proofs given in Section~VI.B extend to  $K$-player, non-zero sum, average reward unichain stochastic games. Unichain property is a mild regularity assumption compared to other regularity conditions such as ergodicity or irreducibility \cite{bhatnagar2009natural}.
\end{remark}%}
\section{Simulations}\label{sec:Simulations}
In this section we test Algorithm~\ref{algo} on a real-world attack dataset corresponding to a ransomware attack. We first provide a brief explanation on the dataset and the extraction of the IFG from the dataset. Then we explain the choice of parameters used in our simulations and present the simulation results.

The dataset consists of system logs with both benign and malicious information flows recorded in a \emph{Linux} computer threatened by a ransomware attack. The goal of the ransomware attack is to open and read all the files in the $./home$ directory of the victim computer and delete all of these files after writing them into an encrypted file named $ransomware.encrypted$. System logs were recorded by RAIN system \cite{JiLeeDowWanFazKimOrsLee-17} and the targets of the ransomware attack (destinations) were annotated in the system logs. Two network sockets that indicate series of communications with external IP addresses in the recorded system logs were identified as the entry points of the attack. The attack consists of three stages, where stage $1$ correspond to privilege escalation, stage $2$ relate to lateral movement of the attack, and stage $3$ represent achieving the goal of encrypting and deleting $./home$ directory.  Immediate conversion of the system logs resulted in an information flow graph, $\bar{\G}$, with $173$ nodes and $2426$ edges.

The attack related subgraph was extracted from $\bar{\G}$ using the following graph pruning steps. 
\begin{enumerate}
	\item For each pair of nodes in $\bar{\G}$ (e.g., process and file, process and process), collapse any existing multiple edges between two nodes to a single directed edge representing the direction of the collapsed edges.
	% which represent multiple system events with the same information flow direction (e.g., write and open, receive and read)
	\item Extract all the nodes in $\bar{\G}$   that have at least one information flow path from an entry point of the attack to a destination of stage one of the attack. 
	\item Extract all the nodes in $\bar{\G}$ that have at least one information flow path from a destination of stage $j$ to a  destination of a stage $j'$, for all $j, j' \in \{1, \ldots, M\}$ such that $j \neq j'$.
	\item From $\bar{\G}$, extract the subgraph corresponding to the entry points, destinations, and the set of nodes extracted in steps $2)$ and $3)$.
	\item Combine all the file-related nodes in the extracted subgraph corresponding to a directory into a single node $(e.g., ./home, ./user)$ in the victim's computer.
	\item If the resulting subgraph contains any cycles use \emph{node versioning} techniques \cite{milajerdi2019holmes} to remove cycles while preserving the information flow dependencies in the graph.
\end{enumerate}
The resulting graph is called as the pruned IFG. %Steps $2)$ and $3)$ are efficiently done by solving a network flow problem with specified source-sink pairs.  
The pruned IFG corresponding to the ransomware attack contains $18$ nodes and $29$ edges (Figure~\ref{fig:IFG}). 
%The pruned IFG  illustrated in Figure~\ref{fig:IFG} is used to define the state space of the DIFT-APT game corresponding to the ransomware attack.
\begin{figure}[h]
	\centering
\includegraphics[width=5.3cm]{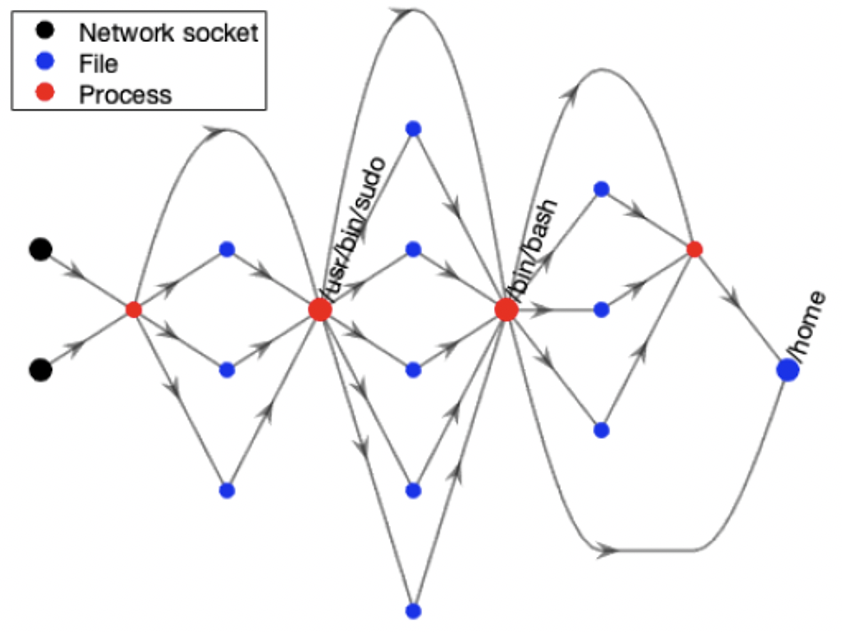}
\caption{ \small IFG of ransomware attack. Nodes of the graph are color coded to illustrate their respective types (network socket, file, and process). Two network sockets are identified as the entry points of the ransomware attack. Destinations of the attack ($/usr/bin/sudo, /bin/bash, /home$) are labeled in the graph.}\label{fig:IFG}
\end{figure}
Simulations use the following cost, reward, and penalty parameters. 
Cost parameters: for all $s^j_i\in {\bf S}$ such that $u_{i} \notin \D_{j}$,  $\C_{D}(s_{i}^{j} ) = -1$ for $j = 1$, $\C_{D}(s_{i}^{j}) = -2$ for $j = 2$, and $\C_{D}(s_{i}^{j} ) = -3$ for $j = 3$. For all other states, $s \in {\bf S}$, $\C_{D}(s) = 0$. Rewards: $\alpha_{\sD}^{1} = 40$, $\alpha_{\sD}^{2} = 80$, $\alpha_{\sD}^{3} = 120$, $\beta_{\sA}^{1} = 20$, $\beta_{\sA}^{2} = 40$, $\beta_{\sA}^{3} = 60$, $\sigma_{\sD}^{1} = 30$, $\sigma_{\sD}^{2} = 50$, and $\sigma_{\sD}^{3} = 70$. Penalties: $\alpha_{\sA}^{1} = -20$, $\alpha_{\sA}^{2} = -40$, $\alpha_{\sA}^{3} = -60$, $\beta_{\sD}^{1} = -30$, $\beta_{\sD}^{2} = -60$, $\beta_{\sD}^{3} = -90$, $\sigma_{\sA}^{1} = -30$, $\sigma_{\sA}^{2} = -50$, and $\sigma_{\sA}^{3} = -70$.
Learning rates used in the simulations are: $\delta_{v}^{n}  =  \delta_{\epsilon}^{n} = 0.5$ if $n < 7000$ and $\delta_{v}^{n}  =  \delta_{\epsilon}^{n} = \frac{1.6}{\kappa(s,n)}$, otherwise. $\delta_{\rho} = \delta_{\pi}^{n} = 1$, if $n < 7000$ and $\delta_{\rho} = \frac{1}{1+\tau(n)\log(\tau(n))}$, $\delta_{\pi}^{n} = \frac{1}{\tau(n)}$, otherwise.

% \begin{eqnarray*}
% (\delta_{v}^{n}  =  \delta_{\epsilon}^{n}, \delta_{\rho}, \delta_{\pi}^{n}) \hspace{-1mm}=\hspace{-1mm}
% \begin{cases}
% 	\begin{array}{lll}
% 		(0.5, 1, 1),~\mbox{~if~} n < 7000\\
% 		(\frac{1.6}{\kappa(s,n)}, \frac{1}{1+\tau(n)\log(\tau(n))}, \frac{1}{\tau(n)}),~\mbox{~otherwise~} 
% 	\end{array}
% \end{cases} 
% %\\
% %\delta_{\rho}^{n} &=&
% %\begin{cases}
% %	\begin{array}{lll}
% %		1, & \mbox{~if~} n < 7000\\
% % 		\frac{1}{1+\tau(n)\log(\tau(n))}, & \mbox{~otherwise~} 
% % 	\end{array}
% % \end{cases} \\
% % \delta_{\pi}^{n} &=&
% % \begin{cases}
% % 	\begin{array}{lll}
% % 		1, & \mbox{~if~} n < 7000\\
% %		\frac{1}{\tau(n)}, & \mbox{~otherwise}.
% %	\end{array}
% %\end{cases}
% \end{eqnarray*}

Note that the learning rates remain constant until iteration $7000$ and then start decaying. We observed that setting learning rates in this fashion helps the finite time convergence of the algorithm. Here, the term $\kappa(s,n)$ in  $\delta_{v}^{n}$ and  $\delta_{\epsilon}^{n}$ denotes the total number of times a state $s \in {\bf S}$ is visited from $7000^\text{th}$ iteration onwards in Algorithm~\ref{algo}. Hence, in our simulations, the learning rates $\delta_{v}^{n}$  of  $v^{n}_{k}(s)$ iterates and $\delta_{\epsilon}^{n}$ of the $\epsilon_{k}^{n+1}(s,a_{k})$ iterates depend on the iteration $n$ and the state visited at iteration $n$. The term $\tau(n) = n - 6999$.

\begin{figure}[h]
	\centering
	\includegraphics[width=7cm]{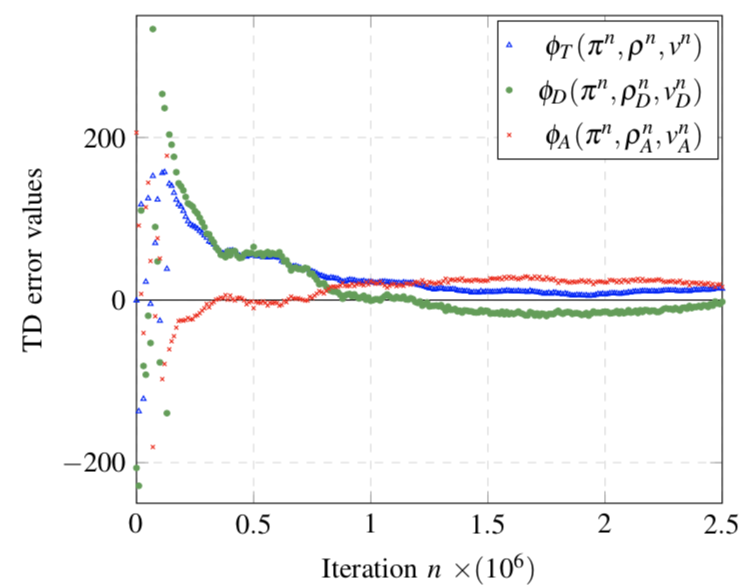}
		\vspace{-3mm}
	\caption{ \small Plots of total Temporal Difference error (TD error), $\phi_{T}(\pi^{n}, \rho^{n}, v^{n})$, DIFT's TD error $\phi_{D}(\pi^{n}, \rho_{D}^{n}, v_{D}^{n})$, and APT's TD error $\phi_{A}(\pi^{n}, \rho_{A}^{n}, v_{A}^{n})$ evaluated at iterations $n = 1, 500, 1000, \ldots, 2.5 \times 10^{6}$ of Algorithm~\ref{algo} for ransomware attack. Values of TD errors at the $n^{\text{th}}$ iteration depend on the policies $\pi^{n}$, average rewards $\rho^{n}$, and value functions $v^{n}$ of DIFT and APT at the $n^{\text{th}}$ iteration. $\phi_{D}(\pi^{n}, \rho_{D}^{n}, v_{D}^{n})$ and $\phi_{A}(\pi^{n}, \rho_{A}^{n}, v_{A}^{n})$ are defined in Eqn.~\eqref{eq:phi_definitions} and $\phi_{T}(\pi^{n}, \rho^{n}, v^{n}) = \phi_{D}(\pi^{n}, \rho_{D}^{n}, v_{D}^{n}) + \phi_{A}(\pi^{n}, \rho_{A}^{n}, v_{A}^{n})$.  }    	 
	\label{plt:objective}
\end{figure}

Conditions \eqref{eq:GNScon1} and \eqref{eq:GNScon2} in Corollary~\ref{cor:conditions} are used to validate the convergence of Algorithm~\ref{algo} to an ARNE of the DIFT-APT game. Let $\phi_{T}(\pi, \rho, v) = \phi_{D}(\pi, \rho_{D}, v_{D}) + \phi_{A}(\pi, \rho_{A}, v_{A})$, where $\pi = (\pi_{D}, \pi_{A})$, $\rho = (\rho_{D}, \rho_{A})$, $v = (v_{D}, v_{A})$. Here, $\phi_{k}(\pi, \rho_{k}, v_{k})$, for $k \in \{D, A\}$, is given by
\begin{equation}\label{eq:phi_definitions}
\begin{split}
\phi_{k}(\pi, \rho_{k}, v_{k}) =& 
\sum\limits_{s \in {\bf S}}\sum\limits_{a_{k} \in \A_{k}(s)}\Big(\rho_{k}  + v_{k}(s) - r_{k}(s,a_{k},\pi_{-k})  \\ &- \sum\limits_{s' \in {\bf S}}\mathbf{P}(s'|s,a_{k},\pi_{-k})v_{k}(s')\Big)\pi_{k}(s,a_{k}) = 0
\end{split}
\end{equation}
We refer to $\phi_{T}(\pi, \rho, v)$, $\phi_{D}(\pi, \rho_{D}, v_{D})$, and $\phi_{A}(\pi, \rho_{A}, v_{A})$ as the total Temporal Difference error (TD error) , DIFT's TD error, and APT's TD error, respectively. Then conditions \eqref{eq:GNScon1} and \eqref{eq:GNScon2} in Corollary~\ref{cor:conditions}  together imply that a policy pair forms an ARNE if and only if $\phi_{D}(\pi, \rho_{D}, v_{D}) = \phi_{A}(\pi, \rho_{A}, v_{A}) = 0$. Consequently, at ARNE $\phi_{T}(\pi, \rho, v) = 0$. Figure~\ref{plt:objective} plots $\phi_{T}$, $\phi_{D}$, and $\phi_{A}$  corresponding to the policies given by Algorithm~\ref{algo} at  iterations $n = 1, 500, \ldots, 2.5 \times 10^{6}$. The plots show that $\phi_{T}$, $\phi_{D}$  and $\phi_{A}$ converge very close to $0$ as $n$ increases. 

%{\color{blue}
%Moreover, the convergence of $\phi_{T}(\pi, \rho, v)$ to zero indicates that the value functions of both players converge to their respective true values. Due to the separation of time scales, the value functions (faster time scale) can be viewed as nearly equilibrated during the update of the policy estimates (slower time scale). The convergence of the policy estimates then follows as a consequence of Theorem~VI.12. }
% and hence ARNE polices are achieved. 
%Let $\phi_{D}$  and $\phi_{A}$ be the objective functions of DIFT and APT given in condition \eqref{eq:GNScon2} in Corollary~\ref{cor:conditions}. Also let $\phi_{T} = \phi_{D} + \phi_{A}$. 

Figure~\ref{plt:reward} plots the average reward values of DIFT and APT in Algorithm~\ref{algo} at $n = 1, 500, \ldots, 2.5 \times 10^{6}$. Figure~\ref{plt:reward} shows that $\rho_{D}^{n}$ and $\rho_{A}^{n}$ converge as the iteration count $n$ increases. 
\begin{figure}[h]
	\centering
	\includegraphics[width=7cm]{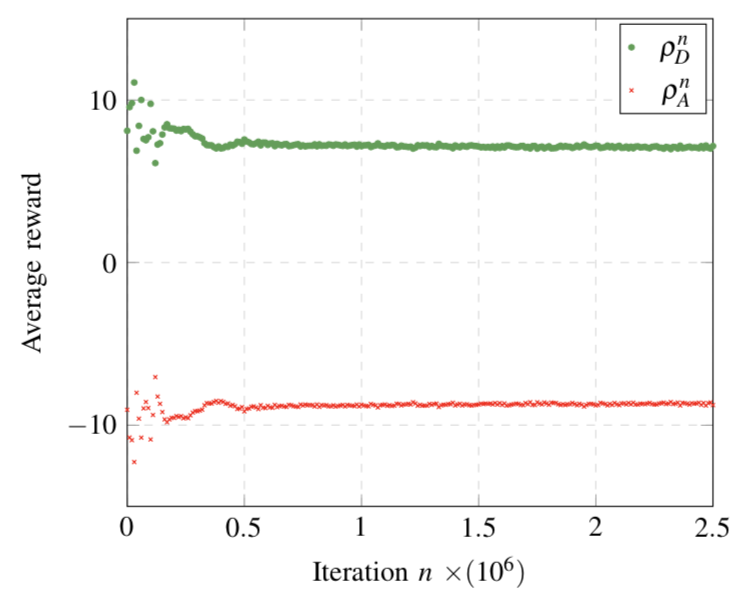}
		\vspace{-3mm}
\caption{\small Plots of the average rewards of DIFT $(\rho_{D}^{n})$ and APT $(\rho_{A}^{n})$ at a iteration $n \in \{1, 500, 1000, \ldots 2.5 \times 10^{6}\}$ of Algorithm~\ref{algo}. Average rewards at the $n^{\text{th}}$ iteration depend on the policies $(\pi^{n})$ of DIFT and APT.}\label{fig:main}
	\label{plt:reward}
\end{figure}

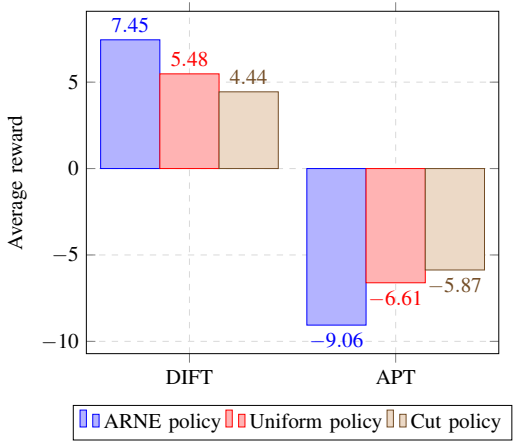
\begin{figure}[H]
	\centering
	\begin{tikzpicture}[scale=0.8]
	\begin{axis}[
	ybar = 0pt,
	bar width =28,
	enlarge x limits=0.5,
	legend style={at={(0.5,-0.15)},
		anchor=north,legend columns=-1},
	ylabel={Average reward},
	symbolic x coords={DIFT,APT},
	xtick=data,
	nodes near coords,
	nodes near coords align={vertical},
	grid=major, % Display a grid
	grid style={dashed,gray!30} % Set the style
	]
	\addplot coordinates {(DIFT,7.45) (APT,-9.06)};
	\addplot coordinates {(DIFT,5.48) (APT,-6.61)};
	\addplot coordinates {(DIFT,4.44) (APT,-5.87)};
	\legend{ARNE policy,Uniform policy,Cut policy}
	\end{axis}
	\end{tikzpicture}
	\caption{\small Comparison of the average rewards of DIFT and APT obtained by the converged policies in Algorithm~\ref{algo} (ARNE policy) against the average rewards of the players obtained by two other policies of DIFT: uniform policy and cut policy. Uniform policy: DIFT chooses an action at every state under a uniform distribution. Cut policy: DIFT performs security analysis at a destination related state, $s^{j}_{i}: u_{i} \in \D_{j}$, with probability one whenever the state of the game is an in-neighbor of that destination related state. }\label{plt:bar}
\end{figure}

Figure~\ref{plt:bar} compares the average rewards of the players corresponding to the converged policies in Algorithm~\ref{algo} (ARNE policy) against the average reward values of the players corresponding to two other policies of DIFT, i)~uniform policy and ii)~cut policy. Note that, in i) DIFT chooses an action at every state under a uniform distribution. Where as in ii)~DIFT performs security analysis at a destination related state, $s^{j}_{i}: u_{i} \in \D_{j}$, with probability one whenever the state of the game is an in-neighbor\footnote{A vertex $u_{i}$ is said to be an in-neighbor of a vertex $u_{i'}$, if there exists an edge $(u_{i}, u_{i'})$ in the directed graph.} of that destination related state.
%at each state is such that if a neighboring state is an attack destination related state, it performs a security analysis with probability one at the neighboring destination  state and otherwise it decides to not to perform a security analysis with probability one.
APT's policy in both uniform policy and cut policy cases are maintained to be as same as in the case of ARNE policy case.  The results show that DIFT achieves a higher average reward under ARNE policy when compared to the uniform and cut policies. Further, results also suggest that APT gets a lower reward under the DIFT's ARNE policy when compared to the uniform and cut policies.   

%Note that the conditions in Corollary~\ref{cor:conditions} can be written as a Nonlinear Program (NLP) with set of variables. The objective function of this NLP will be $\phi_{T}(\pi_{D},\pi_{A}) = $ Then conditions 

\iffalse
\begin{figure}[h]
	\centering
	\begin{tikzpicture}[scale=1]
	\pgfplotsset{compat=1.11,
		/pgfplots/ybar legend/.style={
			/pgfplots/legend image code/.code={%
				\draw[##1,/tikz/.cd,yshift=-0.25em]
				(0cm,0cm) rectangle (3pt,0.8em);},
		},
	}
	\begin{axis}[
	bar width           = 	25pt,
	ybar, 
	x tick label style  = {color=black,text width=1.5cm,align=center, font=\small},
	enlargelimits=0.25,            
	legend style={at={(0.5,1)},
		anchor=north,legend columns=-1},    
	%    legend style={at={(0.5,-0.15)},
	%      anchor=north,legend columns=-1},
	%
	%ymin=-10,
	ylabel={Average reward},
	symbolic x coords={1,2,3},
	xticklabels   = {DIFT, APT},
	xtick=data,
	every node near coord/.append style={font=\tiny},
	nodes near coords align={vertical},
	]
	\addplot coordinates {(1,7.42) (2,-9.43) };
	\addplot coordinates {(1,7.03) (2, -8.74) };
	\addplot coordinates {(1,4.43) (2,-5.89)};
	\legend{{ARNE policy},{Uniform policy},{Cut policy}}
	\end{axis}
	\end{tikzpicture}
	\caption{}\label{fig:bar3}
	\caption{}\label{fig:5}
\end{figure}
\fi

\section{Conclusion}\label{sec:Conclusions}
In this paper we studied the problem of resource efficient and effective detection of Advanced Persistent Threats (APTs) using Dynamic Information Flow Tracking (DIFT) detection mechanism.
We modeled the strategic interactions between DIFT and APT as a nonzero-sum, average reward stochastic game. Our game model captures the security costs, false-positives, and false-negatives associated with DIFT to enable resource efficient and effective defense policies. Our model also incorporates the information asymmetry between DIFT and APT that arises from DIFT's inability to distinguish malicious flows from benign flows and APT's inability to know the locations where DIFT performs a security analysis. Additionally, the game has incomplete information as the transition probabilities (false-positive and false-negative rates) are unknown.
%We obtained the necessary and sufficient conditions required to characterize the Average Reward Nash Equilibrium (ARNE) of DIFT.
We proposed RL-ARNE to learn an Average Reward Nash Equilibrium (ARNE) of the DIFT-APT game. The proposed algorithm is a multiple-time scale stochastic approximation algorithm. We prove the convergence of RL-ARNE algorithm to an ARNE of the DIFT-APT game. 

We evaluated our game model and algorithm on a real-world ransomware attack dataset collected using RAIN framework. Our simulation results showed convergence of the proposed algorithm on the ransomware attack dataset. Further the results showed and validated the effectiveness of the proposed game theoretic framework for devising optimal defense policies to detect APTs. As future work we plan to investigate and model APT attacks by multiple attackers with different capabilities.

\section*{Acknowledgment}
The authors would like to thank Baicen Xiao at Network Security Lab (NSL) at University of Washington (UW) for the discussions on reinforcement learning algorithms.
 
\bibliographystyle{myIEEEtran}      % American Physical Society (APS) style, author-  
\bibliography{MURI_references}

\section*{Appendix}
\begin{lemma}\label{lem:gradient}
	Consider $\Omega_{k, \pi_{-k}}^{s, a_{k}}$ and $\Delta(\pi)$ given in Eqns.~\eqref{eq:Omega} and~\eqref{eq:Delta}, respectively. Then, $\frac{\partial\Delta(\pi)}{\partial\pi_{k}(s,a_{k})} = \sum\limits_{\bar{k} \in \{D,A\}} \Omega_{\bar{k}, \pi_{-k}}^{s, a_{k}}$.
\end{lemma}
\begin{proof}
	Recall $k \in \{D, A\}$ and $-k =  \{D, A\} \setminus k$. 
	\begin{equation*}
	\Delta(\pi) = \sum\limits_{s \in {\bf S}}\left[\sum\limits_{a_{k} \in \A_{k}(s)} \Omega_{k, \pi_{-k}}^{s, a_{k}} \pi_{k}(s,a_{k}) + \sum\limits_{a_{-k} \in \A_{-k}(s)} \Omega_{-k, \pi_{k}}^{s, a_{-k}} \pi_{-k}(s,a_{-k})  \right].
	\end{equation*}
	Taking derivative with respect to $\pi_{k}(s,a_{k})$ in Eqn.~\eqref{eq:Delta} gives, 
	\begin{equation*}
	\frac{\partial\Delta(\pi)}{\partial\pi_{k}(s,a_{k})} =  \Omega_{k, \pi_{-k}}^{s, a_{k}} + \sum\limits_{a_{-k} \in \A_{-k}(s)} \frac{\partial\Omega_{-k, \pi_{k}}^{s, a_{-k}} }{\partial\pi_{k}(s,a_{k})}\pi_{-k}(s,a_{-k})
	\end{equation*}
	From Eqn.~\eqref{eq:Omega}, 
	\begin{equation*}
	\hspace{-2mm}\frac{\partial\Omega_{-k, \pi_{k}}^{s, a_{-k}} }{\partial\pi_{k}(s,a_{k})} \hspace{-0.5mm}= \rho_{-k} \hspace{-0.5mm}+ v_{-k}(s) \hspace{-0.5mm}- \hspace{-0.5mm} r_{-k}(s,a_{k},a_{-k})\hspace{-0.5mm} -\hspace{-1.5mm} \sum\limits_{s' \in {\bf S}}\hspace{-0.5mm}\mathbf{P}(s'|s,a_{k},a_{-k})v_{-k}(s')
	\end{equation*}
	Note that,
	\begin{eqnarray*}
		&\hspace{-5mm}\sum\limits_{a_{-k} \in \A_{-k}(s)} \frac{\partial\Omega_{-k, \pi_{k}}^{s, a_{-k}} }{\partial\pi_{k}(s,a_{k})}\pi_{-k}(s,a_{-k}) = \sum\limits_{a_{-k} \in \A_{-k}(s)} \big[\rho_{k} + v_{k}(s) - \\&\hspace{-5mm}r_{k}(s,a_{k},a_{-k})  -\sum\limits_{s' \in {\bf S}} \mathbf{P}(s'|s,a_{k},a_{-k})v_{k}(s') \big]\pi_{-k}(s,a_{-k}) = \Omega_{-k, \pi_{-k}}^{s, a_{k}}
	\end{eqnarray*}
	Therefore, 
	$\frac{\partial\Delta(\pi)}{\partial\pi_{k}(s,a_{k})} \hspace{-0.5mm}= \sum\limits_{\bar{k} \in \{D,A\}} \Omega_{\bar{k}, \pi_{-k}}^{s, a_{k}}.$
	This proves the result.
\end{proof}
\end{document}